\documentclass[12pt,reqno]{amsart}
\usepackage{amsmath} 

\usepackage{amsthm}
\usepackage{amssymb}
\usepackage{graphics}
\usepackage{latexsym}

\numberwithin{equation}{section}
\newtheorem{thm}{Theorem}[section]
\newtheorem{prop}[thm]{Proposition}
\newtheorem{lem}[thm]{Lemma}
\newtheorem{cor}[thm]{Corollary}

\theoremstyle{remark}
\newtheorem{rem}{Remark}[section]
\newtheorem{defn}{Definition}

\newcommand{\BBB}{\mathbb}
\newcommand{\R}{{\BBB R}}
\newcommand{\Z}{{\BBB Z}}
\newcommand{\T}{{\BBB T}}

\newcommand{\N}{{\BBB N}}

\newcommand{\LR}[1]{{\langle {#1} \rangle }}

\newcommand{\lec}{{\ \lesssim \ }}
\newcommand{\gec}{{\ \gtrsim \ }}

\newcommand{\al}{\alpha}

\newcommand{\ga}{\gamma}
\newcommand{\Ga}{\Gamma}

\newcommand{\vp}{\varphi}

\newcommand{\e}{\varepsilon}

\newcommand{\x}{\xi}
\newcommand{\y}{\eta}

\newcommand{\p}{\partial}
\newcommand{\la}{\lambda}

\newcommand{\de}{\delta}
\newcommand{\om}{\omega}

\newcommand{\supp}{\operatorname{supp}}

\newcommand{\EQ}[1]{\begin{equation} \begin{split} #1
 \end{split} \end{equation}}
\newcommand{\EQS}[1]{\begin{align} #1 \end{align}}
\newcommand{\EQQS}[1]{\begin{align*} #1 \end{align*}}
\newcommand{\EQQ}[1]{\begin{equation*} \begin{split} #1
 \end{split} \end{equation*}}

\newcommand{\F}{\mathcal{F}}

\newcommand{\ti}{\widetilde}
\newcommand{\ha}{\widehat}

\title[L.W.P. of fifth order dispersive equations]{
Parabolic smoothing effect and\\
local well-posedness of fifth order semilinear\\
dispersive equations on the torus
}

\author[K. Tsugawa]{Kotaro Tsugawa}
\address[K. Tsugawa]{Graduate School of Mathematics, Nagoya University,
Chikusa-ku, Nagoya, 464-8602, Japan}
\email[K. Tsugawa]{tsugawa@math.nagoya-u.ac.jp}

\keywords{KdV, modified KdV, fifth order, p-Laplacian, well-posedness, Cauchy problem, energy method, normal form}
\subjclass[2010]{35Q53, 35G55, 37K10, 35A01, 35A02, 35B45, 35B65}
\begin{document}

\begin{abstract}
We consider the Cauchy problem of fifth order dispersive equations on the torus.
We assume that the initial data is sufficiently smooth and the nonlinear term is a polynomial depending on $\p_x^3 u, \p_x^2 u, \p_x u$ and $u$.
We prove that the local well-posedness holds on $[-T,T]$ when the nonlinear term satisfies a condition and  
otherwise, the local well-posedness holds with a smoothing effect only on either $[0,T]$ or $[-T,0]$ and nonexistence result holds on the other time interval, which means that
the nonlinear term can not be treated as a perturbation of the linear part and
the equation has a property of parabolic equations by an influence of the nonlinear term. As a corollary, we also have the same results for $(2j+1)$-st order dispersive equations.
\end{abstract}
\maketitle
\setcounter{page}{001}


\section{Introduction}
We consider the Cauchy problem of fifth order dispersive equations on $\T(:=\R/2\pi\Z)$:
\EQS{
&(\p_t +\ga_0\p_x^5+\ga_1\p_x^3+ \ga_2 \p_x) u(t,x)= N( \p_x^3 u, \p_x^2 u, \p_xu, u), \, (t,x)\in [-T,T] \times \T,\label{e1}\\
&u(0,\cdot)=\vp(\cdot),\label{e2}
}
where the initial data $\vp$, the unknown function $u$ are real valued and $\ga_0, \ga_1$ and $\ga_2$ are real constants with $\ga_0 \neq 0$.
We assume that the nonlinear term $N$ is a polynomial which depends only on $u, \p_x u, \p_x^2 u$ and $\p_x^3 u$ and does not include any constants and linear terms.
That is to say, $N$ can be expressed by
\EQ{
N( \p_x^3 u, \p_x^2 u, \p_xu, u)=\sum_{j=1}^{j_0} N_j(u), \quad
N_j(u)=\la_j(\p_x^3 u)^{a_j}(\p_x^2 u)^{b_j}(\p_x u)^{c_j}u^{d_j}
}
where $\la_j\in \R, j_0\in \N, a_j, b_j, c_j, d_j \in \N\cup \{0\}$ and $p_j:=a_j+b_j+c_j+d_j \ge 2$.
Put $p_{max}:=\max_{1\le j\le j_0} p_j$.
In this paper, we are interested in the case of initial data being sufficiently smooth.
Here, we define a functional $P_N(f)$ to categorize the nonlinear terms.
\begin{defn}
Put
\EQQ{
P_N(f):=\sum_{j=1}^{j_0}P_{N_j}(f), \quad P_{N_j}(f):= \frac{\la_jb_j}{2\pi}\int_\T (\p_x^3 f)^{a_j}(\p_x^2 f)^{b_j-1}(\p_x f)^{c_j}f^{d_j} \, dx.
}
We say that $N$ is non-parabolic resonance type if $P_N \equiv 0$, namely, $P_N(f)=0$ for any $f\in C^\infty(\T)$. Otherwise, we say $N$ is parabolic resonance type.
\end{defn}
Remark that
\EQQ{
P_N(f)=\frac{1}{2\pi}\int_\T \frac{\p}{\p\om_2} N(\om_3,\om_2,\om_1,\om_0)\Big|_{(\om_3,\om_2,\om_1,\om_0)=(\p_x^3 f, \p_x^2 f, \p_x f, f)}\, dx.
}
For instance, put $N_1:=2\p_x^2 u (\p_x u)^2$, $N_2:=(\p_x^2 u)^2 u$ and $N:=N_1+N_2$.
Then, $P_{N_1}(f)=\frac{1}{\pi} \int (\p_x f)^2 \, dx>0, P_{N_2}(f)=\frac{1}{\pi}\int f\p_x^2f \, dx=-P_{N_1}<0$ and $P_N = 0$ for any $f\in C^\infty(\T)$.
Therefore, $N_1$ and $N_2$ are parabolic resonance type and $N$ is non-parabolic resonance type.

Now, we state our main results.
\begin{thm}[L.W.P. for non-parabolic resonance type]\label{thm_nonparabolic}
Let $P_N \equiv 0$, $s\in \N$ and $s \ge 13$. Then, we have the followings.\\
(existence) Let $\vp\in H^s(\T)$. Then, there exist $T=T(\|\vp\|_{H^{12}})>0$ and a solution to \eqref{e1}--\eqref{e2} on $[-T,T]$ satisfying $u\in C([-T,T];H^s(\T))$.\\
(uniqueness) Let $T>0$, $u_1, u_2 \in L^\infty([-T,T];H^{12}(\T))$ be solutions to \eqref{e1}--\eqref{e2} on $[-T,T]$. Then, $u_1(t)=u_2(t)$ on $[-T,T]$.\\
(continuous dependence on initial data) Assume that $\{\vp^j\}_{j\in \N} \subset H^s(\T), \vp\in H^s(\T)$ satisfy $\|\vp^j-\vp\|_{H^s}\to 0$ as $j\to \infty$.
Let $u^j$ (resp. $u$) be the solution obtained above with initial data $\vp^j$ (resp. $\vp$) and $T=T(\|\vp\|_{H^{12}})$.
Then $\sup_{t\in [-T,T]}\|u^j(t)-u(t)\|_{H^s}\to 0$ as $j\to \infty$.
\end{thm}
\begin{thm}[L.W.P. for parabolic resonance type]\label{thm_parabolic1}
Let $P_N\not\equiv 0$, $s\in \N$ and $s \ge 13$. Then, we have the followings.\\
(existence) Let $\vp\in H^s(\T)$ and $P_N(\vp)>0$ (resp. $P_N(\vp)<0$). Then, there exist $T=T(P_N(\vp),\|\vp\|_{H^{12}})>0$ and a solution to \eqref{e1}--\eqref{e2} on $[0,T]$ (resp. $[-T,0]$) satisfying $u\in C([0,T];H^s(\T)) \cap C^\infty((0,T]\times \T)$ and $P_N(u(t)) \ge P_N(\vp)/2$ on $[0,T]$ (resp. $u\in C([-T,0];H^s(\T)) \cap C^\infty([-T,0)\times \T)$ and $P_N(u(t)) \le  P_N(\vp)/2$ on $[-T,0]$).\\
(uniqueness) Let $T>0$, $u_1, u_2 \in L^\infty([0,T];H^{12}(\T))$ (resp.  $u_1, u_2 \in L^\infty([-T,0];H^{12}(\T))$) be solutions to \eqref{e1}--\eqref{e2} and $P_N(u_1(t)) \ge 0$ on $[0,T]$ (resp. $P_N(u_1(t)) \le 0$ on $[-T,0]$). Then, $u_1(t)=u_2(t)$ on $[0,T]$ (resp. $[-T,0]$).\\
(continuous dependence on initial data) Assume that $\{\vp^j\}_{j\in \N} \subset H^s(\T), \vp\in H^s(\T)$ satisfy $P_N(\vp)>0$ (resp. $P_N(\vp)<0$)) and $\|\vp^j-\vp\|_{H^s}\to 0$ as $j\to \infty$.
Let $u^j$ (resp. $u$) be the solution obtained above with initial data $\vp^j$ (resp. $\vp$) and $T=T(P_N(\vp), \|\vp\|_{H^{12}})$.
Then $\sup_{t\in [0,T]}\|u^j(t)-u(t)\|_{H^s(\T)}\to 0$ (resp. $\sup_{t\in[-T,0]}\|u^j(t)-u(t)\|_{H^s(\T)}\to 0$) as $j\to \infty$.
\end{thm}
\begin{thm}[non existence for parabolic resonance type]\label{thm_parabolic2}
Let $\vp \not\in C^\infty(\T)$ and $P_N(\vp)<0$ (resp. $P_N(\vp)>0$). Then, for any small $T>0$, there does not exist any solution to \eqref{e1}--\eqref{e2} on $[0,T]$ (resp. $[-T,0]$) satisfying $u\in C([0,T];H^{13}(\T))$ (resp. $u\in C([-T,0];H^{13}(\T))$). 
\end{thm}
\begin{rem}
Theorem \ref{thm_nonparabolic} is a typical result for dispersive equations in the following sense: they can be solved on both positive and negative time intervals and the regularity of the solution is same as that of initial data.
Theorems \ref{thm_parabolic1} and \ref{thm_parabolic2} are typical results for parabolic equations in the following sense: they can be solved on either positive or negative time interval with strong smoothing effect and they are ill-posed on the other time interval.
Since \eqref{e1} are semilinear dispersive equations, Theorem \ref{thm_nonparabolic} is a natural result. On the other hand, Theorems \ref{thm_parabolic1} and \ref{thm_parabolic2} are somewhat surprising.
These theorems mean that when the nonlinear term is parabolic resonance type, the nonlinear term cannot be treated as a perturbation of the linear part and the effect by the second derivative in the nonlinear part is dominant.
\end{rem}
\begin{rem}
Put $J_{N,\e}^{(j)}=\frac{d^j}{dt^j}P_N(u_\e(t))\big|_{t=0}$ for $\e\in [0,1], j=1,2,\ldots$ and $u_\e$ satisfying \eqref{es1}--\eqref{es2} with $\vp\in H^{3+5j}(\T)$. Note that we can express $J_{N,\e}^{(j)}$ only by $\vp$ if we use \eqref{es1}--\eqref{es2}.
For instance,
\EQQ{
J_{N,\e}^{(1)}=\sum_{j=1}^{j_0}\frac{\la_jb_j}{2\pi}\int_\T &\Big(a_j(\p_x^3 \vp)^{a_j-1}(\p_x^2 \vp)^{b_j-1}(\p_x \vp)^{c_j}\vp^{d_j}\p_x^3\\
&+(b_j-1)\p_x^3 \vp)^{a_j}(\p_x^2 \vp)^{b_j-2}(\p_x \vp)^{c_j}\vp^{d_j}\p_x^2\\
&+c_j(\p_x^3 \vp)^{a_j}(\p_x^2 \vp)^{b_j}(\p_x \vp)^{c_j-1}\vp^{d_j}\p_x\\
&+d_j(\p_x^3 \vp)^{a_j}(\p_x^2 \vp)^{b_j}(\p_x \vp)^{c_j}\vp^{d_j-1}\Big)\\
&\times\big(( -\e\p_x^4-\ga_0\p_x^5-\ga_1\p_x^3- \ga_2 \p_x) \vp+ N(\vp, \p_x\vp, \p_x^2 \vp, \p_x^3 \vp)\big)\, dx.
}
Obviously, $J_{N,\e}^{(j)}\to J_{N,0}^{(j)}$ when $\e\to 0$.
When $N$ is parabolic resonance type and $P_N(\vp)=J_{N,0}^{(1)}=\cdots=J_{N,0}^{(j-1)}=0$, the direction of the existence time depends on the sign of $J_{N,0}^{(j)}$.
Precisely, if $s \ge 13+5j, \vp \in H^s(\T), P_N(\vp)=J_{N,0}^{(1)}=\cdots=J_{N,0}^{(j-1)}=0$ and $J_{N,0}^{(j)}>0$ (resp. $J_{N,0}^{(j)}<0$), then there exists a unique solution $u\in C([0,T];H^s(\T)) \cap C^\infty((0,T)\times \T)$ (resp. $C([-T,0];H^s(\T))\cap C^\infty((-T,0)\times \T)$) and there does not exist any solution $u\in C([-T,0];H^{13+5j}(\T))$ (resp. $C([0,T];H^{13+5j}(\T))$).
We can prove this by using Taylor's expansion instead of the mean value theorem in the proof of Lemma \ref{pre_apriori2} and replacing \eqref{Pest} with
\EQQ{
\inf_{t\in[0,T_+]} \frac{P_N(u_\e(t))}{t^j} \ge  \frac{1}{2}\frac{d^j}{dt^j} P_N(u_\e(t))\Big|_{t=0},
}
the second term of the left-hand side of \eqref{apriori1} with
\EQQ{
J_{N,\e}^{(j)}\int_0^t (t')^j\|\p_x^{s+1}u_\e(t')\|_{L^2}^2\, dt'
}
and \eqref{apriori2} with
\EQQ{
\inf_{t\in[0,T_*]} \frac{P_N(u_\e(t))}{t^j} \ge  \frac{1}{2}J_{N,\e}^{(j)}.
}
\end{rem}

Here, we give some examples which satisfies \eqref{e1}.
The fifth order KdV equation:
\EQ{\label{ex1}
(\p_t +\p_x^5)u -5\p_x(\p_x u)^2 +10 \p_x^2 (u\p_x u)+\p_x(u^3)=0,
}
is the second equation in the KdV hierarchy.
The fifth order modified KdV equation:
\EQ{\label{ex2}
(\p_t +\p_x^5)u +5\p_x(u\p_x^2(u^2)) \pm 6\p_x (u^5)=0
}
is the second equation in the modified KdV hierarchy.
In \cite{Benney}, Benney introduced the following equation to describe interactions between short and long waves:
\EQ{\label{ex3}
(\p_t +\p_x^5) u=u\p_x^3 u+2\p_x u\p_x^2 u.
}
In \cite{Lisher}, Lisher proposed the following equation in the study of anharmonic lattices:
\EQ{\label{ex4}
(\p_t +\p_x^5 +\p_x^3) u
=-\frac{1}{2}u\p_x^3 u-u^2\p_x^3 u-(1+4u)\p_x u\p_x^2 u-(\p_x u)^3-(u+u^2)\p_x u.
}
The nonlinear terms of \eqref{ex1}--\eqref{ex4} are non-parabolic resonance type.
Let $p$ is an odd number greater than $4$. 
\EQ{\label{plap}
(\p_t+\ga_0 \p_x^5)u=\p_x(|\p_x u|^{p-2}\p_x u)
}
is called the parabolic p-Laplacian equation when $\ga_0=0$.
Let $q$ is a natural number greater than $2$.
\EQ{\label{porous}
(\p_t+\ga_0\p_x^5)u=\p_x^2 (u^q)
}
is called the porous medium equation when $\ga_0=0$.
These are degenerated parabolic equations. The second derivative with $x$ of \eqref{plap} (resp. \eqref{porous}) vanished at the point such that $\p_xu(t,x) =0$ (resp. $u(t,x)=0$).
Therefore, it does not have the parabolic smoothing effect when the initial data $\vp$ satisfy $\p_x\vp(x)=0$ (resp. $\vp(x)=0$) at some points $x\in \T$.
However, when $\ga_0 \neq 0$, \eqref{plap} and \eqref{porous} are the fifth order dispersive equations with the parabolic resonance type nonlinearity.
Therefore, they have the parabolic smoothing effect by Theorem \ref{thm_parabolic1}. This means that the dispersion recovers the smoothing effect of the degenerated parabolic equations.

For the case of $x\in \R$, there are many results related to fifth order dispersive equations (\cite{Gru}, \cite{Zihua}, \cite{Kwon1}, \cite{Kwon2}, \cite{TKK}, \cite{KPilod}, \cite{KPV}, \cite{Ponce}, \cite{Tomoeda}).
In \cite{KPV}, Kenig, Ponce and Vega consider the following $(2j+1)$-st order dispersive equations:
\EQQ{
(\p_t+\p_x^{2j+1})u=N(\p_x^{2j}u,\ldots,\p_x u,u).
}
Employing the gauge transformation introduced by Hayashi \cite{Hayashi}, Hayashi and Ozawa \cite{HayashiOzawa1}, \cite{HayashiOzawa2} and the smoothing effect for the linear part:
\EQQ{
\|\p_x^j e^{t\p_x^{2j+1}}\vp\|_{L^\infty_xL^2_t}\lec \|\vp\|_{L^2},
}
they proved the local well-posedness in $H^{s_1}(\R)\cap H^{s_2}(\R;x^2\,dx)$ for sufficiently large integers $s_1, s_2$.
The result means that the local solution is controlled by the linear part of the equation in the case $x\in \R$ unlike the parabolic resonance type in the case $x\in \T$. 
In \cite{Kwon2}, Kwon proved the local well-posedness of
\EQ{\label{kwonfifth}
(\p_t+\p_x^5)u=c_1\p_x u\p_x^2 u+c_2 u\p_x^3 u
}
in $H^s(\R)$ for $s>5/2$.
The standard energy method gives only the following:
\EQ{\label{energyineq0}
\frac{d}{dt}\|\p_x^su(t)\|_{L^2}^2 \lec \|\p_x^3u\|_{L^\infty}\|\p_x^s u(t)\|_{L^2}+\Big|\int_{\R} \p_x u\p_x^{s+1}u\p_x^{s+1}u\, dx\Big|.
}
It is the main difficulty in this problem that the last term can not be estimated by $\|u(t)\|_{H^s}$. To overcome the difficulty, Kwon introduced the following energy:
\EQ{\label{menergy0}
E_{*}(u(t)):=\|D^su(t)\|_{L^2}^2+\|u(t)\|_{L^2}^2+C_s\int_\R u(t)D^{s-2}\p_xu(t)D^{s-2}\p_x u(t),
}
where $D:=\F^{-1} |\x| \F_x$.
The last term is the correction term and used to cancel out the last term in \eqref{energyineq0}.

For the case of $x\in \T$, the linear part does not have any smoothing effect and only a few results are known.
In \cite{Sa}, Saut proved the existence of global weak solutions to nonlinear $(2j+1)$-st order dispersive equations which have Hamiltonian structure.
In \cite{Schwarz}, Schwarz Jr.~ proved the unique existence of the global solution to the higher order KdV equations of the member of the KdV hierarchy.
In \cite{Kwak}, Kwak proved the global well-posedness of the fifth order KdV equation \eqref{ex1} in $H^s(\T)$ for $s\ge 2$.
All these results require some special structure to the nonlinear terms and parabolic resonance type nonlinearities are excluded.

Our proof based on the method by Kwon in \cite{Kwon2} mentioned above (see also Segata \cite{Segata} and Kenig-Pilod \cite{KPilod2}).
Since the nonlinear term depends on $\p_x^2 u$ and $\p_x^3 u$,
the standard energy method does not work.
Therefore, we employ the following energy:
\EQ{\label{menergy}
E_{s}(u):=& \frac{1}{2}\|\p_x^s u\|_{L^2}^2+\frac{1}{2}\|u\|_{L^2}^2\Big(1+C_s\sum_{j=1}^{j_0} \|u\|_{H^4}^{s(p_j-1)}\Big)\\
+& \sum_{j=1}^{j_0}{\Ga}^{(p_j)}\Big( \frac{(ik_{p_j})^{s+1}(ik_{p_j+1})^{s+1}M_{NR,j}}{\Phi^{(p_j)}}; u,\ldots,u \Big).
}
See section 2 for the definitions of $\Ga^{(p_j)}, M_{NR,j}$ and $\Phi^{(p_j)}$.
We will show the following energy inequality for sufficiently smooth solutions (see Corollary \ref{cor_energy_est}):
\EQ{\label{energyineq}
\frac{d}{dt} E_s(u(t))+P_N(u(t))\| \p_x^{s+1}u(t) \|_{L^2}^2 \lec E_s(u(t))(1+E_{8}(u(t)))^{r(s)}.
}
Note that the second term of the left hand side has the parabolic smoothing effect when $P_N(u(t)) \neq 0$.
The difference between our proof and the proof by Kwon is how to define the correction term of the energy and the presence of the term having the parabolic effect in the energy inequality.
Remark that the nonlinear term of \eqref{kwonfifth} is non-parabolic resonance type and quadratic.
As mentioned in Remark \ref{oscillation_rem}, it seems difficult to find the resonance part exactly when $p_j \ge 4$.
Therefore, we use Lemma \ref{oscillation_est} and the normal form reduction to recover the derivative losses. Then, naturally we obtain the correction term of \eqref{menergy} and the second term in the left-hand side of \eqref{energyineq} (see Proposition \ref{multi_est2}).
The normal form reduction has been applied for many nonlinear dispersive equations to study the global behavior of solution with small initial data and unconditional uniqueness of local solutions (see e.g. \cite{Babin}, \cite{GMS}, \cite{Zihua2}, \cite{Shatah}).

Here, we give a generalization of the main theorems. We have the same results for the following $(2j+1)$-st order dispersive equations:
\EQS{
\begin{split}
&(\p_t +\ga_0\p_x^{2j+1}+\ga_1\p_x^{2j-1}+\cdots+\ga_{j-1}\p_x^3+ \ga_j \p_x) u(t,x)\\
= &N( \p_x^3 u, \p_x^2 u, \p_xu, u), \ \ \ \ (t,x)\in [-T,T] \times \T,\label{he1}
\end{split}
}
where $j\ge 3$, $\ga_0, \ga_1,\ldots, \ga_j$ are real constants with $\ga_0 \neq 0$.
Since the nonlinear term $N$ of \eqref{he1} is exactly same as \eqref{e1}, it has the same difficulty with the derivative loss.
On the other hand, the dispersive effect of the linear part of \eqref{he1} is stronger than that of \eqref{e1}. Therefore, the normal form reduction works better (see \eqref{ineq_oscillation_higher}) and we easily obtain the following result as a corollary of Theorems \ref{thm_nonparabolic}, \ref{thm_parabolic1} and \ref{thm_parabolic2}.
\begin{cor}\label{thm_higher_order}
Theorems \ref{thm_nonparabolic}, \ref{thm_parabolic1} and \ref{thm_parabolic2} hold even if we replace \eqref{e1} with  \eqref{he1}.
\end{cor}
Finally, we give a remark on the KdV hierarchy.
We consider the $(2j+1)$-st order dispersive equations in the following form:
\EQ{\label{high_eqs}
(\p_t+\p_x^{2j+1})u(t,x)=\sum_{k=2}^{j+1}N_{j,k}(u), \ \ N_{j,k}(u)=\sum_{|l|=2(j-k)+3}\ga_{j,k,l}\p_x^{l_0}\prod_{i=1}^k \p_x^{l_i}u,
}
where $|l|=l_0+\cdots+l_k, l_0 \ge 1, l_i \in \N \cup \{0\}$ for $i=1,\ldots,k$
and $\ga_{j,k,l} \in \R$.
We define the rank $r$ of monomial $\p_x^{l_0}\prod_{i=1}^k \p_x^{l_i}u$ by $r=k+|l|/2$ where $k$ is the number of factors (degree) and $|l|$ is the total number of differentiations (derivative index).
Note that the rank of all monomials contained in the nonlinear term of \eqref{high_eqs} equals $j+3/2$ and the nonlinear term is in divergence form.
When the coefficients $\ga_{j,k,l}$ satisfy a condition, \eqref{high_eqs} is
the $(2j+1)$-st order KdV equation of the member of the KdV hierarchy (see \cite{Gru}, \cite{Schwarz} for more details). For general $\ga_{j,k,l}$, \eqref{high_eqs} is not an integrable system and is not even a Hamiltonian system.
In \cite{Gru}, Gr\"unrock proved the local well-poedness of \eqref{high_eqs}  in $\ha{H}_s^r(\R)$ for general $\ga_{j,k,l}$ when $x\in \R$.
In contrast, the local well-posedness does not always hold for general $\ga_{j,k,l}$ when $x\in \T$.
For instance, we have non-existence result for
\EQQ{
(\p_t+\p_x^{11})u(t,x)=\p_x(\p_x u)^4
}
when $x\in \T$ by Corollary \ref{thm_higher_order} because $P_N(u)=\frac{2}{\pi}\int_\T (\p_x u)^3\, dx \not\equiv 0$ when $N=\p_x(\p_x u)^4$.
This means that the divergence form and the rank are not enough and we need to assume a condition on $\ga_{j,k,l}$ to prove the local well-posedness of \eqref{high_eqs} when $x\in \T$.
In fact, Schwarz Jr.~ used a property of complete integrability to show the unique existence of the solution to the higher order KdV equations of the member of the KdV hierarchy in \cite{Schwarz}.

The rest of the paper is organized as follows.
In Section 2, we give some notations, preliminaries and show the local well-posedness of
the following parabolic regularized equation:
\EQS{
\begin{split}
&(\p_t +\e\p_x^4+\ga_0\p_x^5+\ga_1 \p_x^3+ \ga_2 \p_x) u_\e(t,x)\\
&\hspace*{11em}= N(\p_x^3 u_\e, \p_x^2 u_\e, \p_xu_\e, u_\e), \quad (t,x)\in [0,T_\e) \times \T,
\end{split}\label{es1}\\
&u_\e(0,\cdot)=\vp(\cdot),\label{es2}
}
when $\e \in (0,1]$. 
In Section 3, we show the energy inequality for the difference of two solutions to \eqref{es1}, which is the main estimate in this paper.
In Section 4, we show the main theorems by combining the local well-posedness for the regularized equation (Proposition \ref{para_exist}),
the comparison lemma between the Sobolev norm and
the energy (Lemma \ref{sobolev_energy}) and the energy inequality (Proposition \ref{thm_energy_est})
with Bona-Smith's approximation argument.
At the end of Section 4, we give an outline of the proof of Corollary \ref{thm_higher_order}.

\section*{Acknowledgement}
This research was supported by KAKENHI (25400158).


\section{Notations, Preliminaries and Parabolic regularized equation}
First, we give some notations.
The Fourier transform and the Sobolev norm are defined as below:
\EQQ{
&\F[f](k):=\ha{f}(k):=\frac{1}{2\pi}\int_\T f(x)e^{-ikx} \, dx, \ \ \
\F^{-1}[\ha{f}](x):=\sum_{k=-\infty}^{k=\infty}\ha{f}(k)e^{ikx},\\
&\|f\|_{H^s}:=\| \LR{k}^s \ha{f}(k) \|_{\ell^2}, \ \ \LR{k}:=(1+|k|^2)^{1/2}.
}
By the Plancherel theorem, for the $L^2$-inner product, we have
\EQQ{
(f,g)_{L^2}=\frac{1}{2\pi}\int_\T fg \, dx=\sum_{k\in \Z} \ha{f}(k)\ha{g}(-k)=(\ha{f}, \ha{g})_{\ell^2}.
}
Put $\Z_0^{(p)}:=\{(k_1,\ldots,k_p,k_{p+1}) \in \Z^{p+1} \, | \, k_{(1,p+1)}=0\}$ where $k_{(l,m)}$ means $\sum_{j=l}^m k_j$.
Put $\phi(k):=i(\ga_0k^5-\ga_1k^3+\ga_2k)$ and $\Phi^{(p)}(\vec{k}^{(p)}):=-\sum_{l=1}^{p+1}\phi(k_l)=\phi(k_{(1,p)})-\sum_{l=1}^{p}\phi(k_l)$ for $\vec{k}^{(p)}:=(k_1,\ldots,k_p,k_{p+1}) \in \Z_0^{(p)}$.
\begin{defn}
A multiplier is a function on $\Z_0^{(p)}$.
For a multiplier $M$ and functions $\{f_l\}_{l=1}^{p+1}$ on $\T$, we define
multilinear operators:
\EQQ{
&\Ga^{(p)}(M;f_1,\ldots,f_p,f_{p+1}):=\F^{-1} \Big[\sum_{\vec{k}^{(p)}\in \Z_0^{(p)}} M(\vec{k}^{(p)})\prod_{l=1}^{p+1}\ha{f_l}(k_l)\Big].
}
\end{defn}
Put
\EQQ{
D_{a,b,c}:=\prod_{l=1}^a(ik_l)^3\prod_{l=a+1}^{a+b}(ik_l)^2\prod_{l=a+b+1}^{a+b+c}(ik_l), 
}
which is used to describe nonlinear terms. For the $L^2$-inner product $(\cdot,\cdot)_{L^2}$, it follows that
\EQ{
&(N_j(f),g)_{L^2}=\Ga^{(p)}(\la_jD_{a_j,b_j,c_j};f,\ldots,f,g).\label{eq2.1}
}
For an integer $p \ge 2$, we define multipliers $M_H^{(p)}$ and $M_{NZ}^{(p)}$ as below:
\EQQS{
&M_{H}^{(p)}(\vec{k}^{(p)}):=
\begin{cases}
1, \,\, \text{when}\,\, \min\{|k_p|^{4/5},
 |k_{p+1}|^{4/5}\} \ge C \max \{|k_1|,\ldots,|k_{p-1}|\},\\
0, \,\, \text{otherwise},
\end{cases}\\
&M_{NZ}^{(p)}(\vec{k}^{(p)}):=
\begin{cases}
1, \,\, \text{when}\,\, k_{(1,p-1)}\neq 0,\\
0, \,\, \text{when}\,\, k_{(1,p-1)}=0,
\end{cases}
}
where $C$ in the definition of $M_H^{(p)}$ is a sufficiently large constant, which is determined by Lemma \ref{oscillation_est} and Lemma \ref{oscillation_est2}.
\begin{lem}\label{lem_MH}
Let $\vec{k}^{(p)}\in \supp \, (1-M_H^{(p)})\subset \Z_0^{(p)}$.
Then,
\EQQ{
\max\{|k_{p}|, |k_{p+1}|\} \lec \max_{1\le l\le p-1} |k_l|^{5/4}.
}
\end{lem}
\begin{proof}
By symmetry, we  only need to consider the case $|k_{p}| \le |k_{p+1}|$.
When $\vec{k}^{(p)}\in \supp \, (1-M_H^{(p)})\subset \Z_0^{(p)}$,
it follows that $|k_{p}| \lec \max_{1\le l\le p-1}|k_l|^{5/4}$.
Thus, $|k_{p+1}|=|k_{(1,p)}|\lec \max_{1\le l\le p-1} |k_l|^{5/4}$.
\end{proof}
Put
\EQQ{
M_{NR,j}:=M_{H}^{(p_j)}M_{NZ}^{(p_j)}\la_j\Big((s-3/2)a_j(ik_{(1,p_j-1)})D_{a_j-1,b_j,c_j}+b_j  D_{a_j,b_j-1,c_j}\Big).
}
Obviously, it follows that
\EQ{\label{eq_mnr}
|M_{NR,j}| \lec |k_{(1,p_j-1)}|\prod_{l=1}^{p_j-1} \LR{k_l}^3,
}
where the implicit constant depends on $s$. 
\begin{lem}\label{byparts_lem}
Let $M^{(p)}(\vec{k}^{(p)})$ be symmetric with respect to $k_p$ and $k_{p+1}$.
Then,
\EQ{
\Ga^{(p)}\big((ik_{p}\big)M^{(p)};f,\ldots,f,g,g)=&\Ga^{(p)}\big((ik_{p+1}\big)M^{(p)};f,\ldots,f,g,g)\\
=&-\frac{1}{2}\Ga^{(p)}\big((ik_{(1,p-1)})M^{(p)};f,\ldots,f,g,g\big),\label{sym1}
}
\EQ{
\Ga^{(p)}\big((ik_{p})^3M^{(p)};f,\ldots,f,g,g\big)=&-\frac{1}{2}\Ga^{(p)}\big((ik_{(1,p-1)})^3M^{(p)};f,\ldots,f,g,g\big)\\
+\frac{3}{2}\Ga^{(p)}\big((&ik_{(1,p-1)})(ik_p)(ik_{p+1})M^{(p)};f,\ldots,f,g,g\big).\label{sym2}
}

\end{lem}
\begin{proof}
By symmetry,
\EQQ{
{\Ga}^{(p)}((ik_{p})M^{(p)},f,\ldots,f,g,g)&={\Ga}^{(p)}\big((ik_{p+1})M^{(p)},f,\ldots,f,g,g\big)\\
&={\Ga}^{(p)}\Big( \frac{(ik_p+ik_{p+1})M^{(p)}}{2},f,\ldots,f,g,g\Big).
}
Since $ik_p+ik_{p+1}=-ik_{(1,p-1)}$ for $\vec{k}^{(p)}\in \Z_0^{(p)}$,
we obtain \eqref{sym1}. By symmetry,
\EQQ{
{\Ga}^{(p)}((ik_{p})^3M^{(p)},f,\ldots,f,g,g)&={\Ga}^{(p)}\big((ik_{p+1})^3M^{(p)},f,\ldots,f,g,g\big)\\
&={\Ga}^{(p)}\Big( \frac{((ik_p)^3+(ik_{p+1})^3)M^{(p)}}{2},f,\ldots,f,g,g\Big).
}
Since $(ik_p)^3+(ik_{p+1})^3=-(ik_{(1,p-1)})^3+3(ik_p)(ik_{p+1})(ik_{(1,p-1)})$
for $\vec{k}^{(p)}\in \Z_0^{(p)}$, we obtain \eqref{sym2}.
\end{proof}
The following lemma is the Gagliardo-Nirenberg inequality for periodic functions.
For the proof, see Section 2 in \cite{Schwarz}.
\begin{lem}\label{lemGN}
Assume that integers $l$ and $m$ satisfy $0\le l \le m-1$ and a real number $p$ satisfies $2\le p \le \infty$.
Put $\al=(l+1/2-1/p)/m$. Then, we have
\EQQ{
\|\p_x^l f\|_{L^p} \lec 
\begin{cases}
\|f\|_{L^2}^{1-\al}\|\p_x^m f\|_{L^2}^\al, \qquad ( \text{when}\,\, 1\le l\le m-1),\\
\|f\|_{L^2}^{1-\al}\|\p_x^m f\|_{L^2}^\al +\|f\|_{L^2}, \qquad ( \text{when} \,\, l=0),
\end{cases}
}
for any $f\in H^m(\T)$.
Particularly, $\|\p_x^l f\|_{L^p} \lec \|f\|_{L^2}^{1-\al}\|f\|_{H^m}^\al$.
\end{lem}
\begin{lem}\label{lem_ga4}
(i) Let $m,p\in \N$, $m\le p$, $a,b,c\in \N\cup\{0\}$ and $M^{(p)}(\vec{k}^{(p)})$ be a multiplier such that
\EQQ{
|M^{(p)}(\vec{k}^{(p)})|\lec \LR{k_{p+1}}^c\max_{m\le l\le p}\LR{k_l}^{b} \prod_{l=1}^{p}\LR{k_l}^a.
}
Then, it follows that
\EQ{\label{ee3.41}
&\big|{\Ga}^{(p)}\big( M^{(p)};f,\ldots,f,h\big)
-{\Ga}^{(p)}\big( M^{(p)};f,\ldots,f,g,\ldots,g,h\big)\big|\\
\lec& \|h\|_{H^{c}}(\|f\|_{H^{a+1}}+\|g\|_{H^{a+1}})^{p-2}\\
\times&\{ \|f-g\|_{H^{a+b}}(\|f\|_{H^{a+1}}+\|g\|_{H^{a+1}})+\|f-g\|_{H^{a+1}}(\|f\|_{H^{a+b}}+\|g\|_{H^{a+b}})\},
}
where the components from the first to $(m-1)$-st are $f$ and the components from the $m$-th to $p$-th are $g$ in the second term of the left-hand side of \eqref{ee3.41}.\\
(ii) Let $m,p\in \N$, $m\le p-1$, $a,b,c,d\in \N\cup\{0\}$ and $M^{(p)}(\vec{k}^{(p)})$ be a multiplier such that
\EQQ{
|M^{(p)}(\vec{k}^{(p)})|\lec \LR{k_{p}}^d\LR{k_{p+1}}^c\max_{m\le l\le p-1}\LR{k_l}^{b} \prod_{l=1}^{p-1}\LR{k_l}^a.
}
Then, it follows that
\EQ{\label{ee3.412}
&\big|{\Ga}^{(p)}\big( M^{(p)};f,\ldots,f,h_1,h_2\big)
-{\Ga}^{(p)}\big( M^{(p)};f,\ldots,f,g,\ldots,g,h_1,h_2\big)\big|\\
\lec& \|h_1\|_{H^{d}}\|h_2\|_{H^{c}}(\|f\|_{H^{a+1}}+\|g\|_{H^{a+1}})^{p-3}\\
\times&\{ \|f-g\|_{H^{a+b+1}}(\|f\|_{H^{a+1}}+\|g\|_{H^{a+1}})\\
&+\|f-g\|_{H^{a+1}}(\|f\|_{H^{a+b+1}}+\|g\|_{H^{a+b+1}})\}
}
where the components from the first to $(m-1)$-st are $f$ and the components from the $m$-th to $(p-1)$-st are $g$ in the second term of the left-hand side of \eqref{ee3.412}.
\end{lem}
\begin{proof}
Since
\EQQ{
&{\Ga}^{(p)}\big( M^{(p)};f,\ldots,f,h\big)
-{\Ga}^{(p)}\big( M^{(p)};f,\ldots,f,g,\ldots,g,h\big)\\
=&{\Ga}^{(p)}\big( M^{(p)};f,\ldots,f,f-g,h\big)
  +{\Ga}^{(p)}\big( M^{(p)};f,\ldots,f,f-g,g,h\big)\\
&+\cdots+{\Ga}^{(p)}\big( M^{(p)};f,\ldots,f,f-g,g,\ldots,g,h\big),
}
by the Sobolev inequality, we get \eqref{ee3.41}.
In the same  manner we have \eqref{ee3.412}.
\end{proof}
\begin{lem}\label{oscillation_est}
Let $p\ge 2$, $|k_{p}|^{4/5}\ge C\max_{1\le l \le p-1}\{|k_l|\}$ and $C$ be sufficiently large. Then,
\EQ{\label{ineq_oscillation}
|\Phi^{(p)}(\vec{k}^{(p)})| \gec |k_p|^4|k_{(1,p-1)}|\sim|k_{p+1}|^4|k_{(1,p-1)}|
}
for $\vec{k}^{(p)} \in \Z_0^{(p)}$.
\end{lem}
\begin{rem}\label{oscillation_rem}
We have the following factorization formulas:
\EQQ{
&\Phi^{(2)}=i\frac{5}{2}k_1k_2(k_1+k_2)(k_1^2+k_2^2+(k_1+k_2)^2),\\
&\Phi^{(3)}=i\frac{5}{2}(k_1+k_2)(k_2+k_3)(k_3+k_1)\big((k_1+k_2)^2+(k_2+k_3)^2+(k_3+k_1)^2\big).
}
Therefore, we can easily solve
\EQ{\label{factorization}
\Phi^{(p)}(\vec{k}^{(p)})=0
}
for $\vec{k}^{(p)} \in \Z_0^{(p)}$ when $p=2$ or $3$.
On the other hand, no factorization formula is known for $p\ge 4$ and
the distribution of the solutions seems to be complicated.
For instance, $\vec{k}^{(p)}=(24,28,67,-3,-54)$ satisfies \eqref{factorization} when $p=5$.
Therefore, it seems difficult to solve \eqref{factorization} when $p\ge 4$.
\end{rem}
\begin{proof}[Proof of Lemma \ref{oscillation_est}]
It is obvious when $k_{(1,p-1)}=0$.
Therefore, we assume $|k_{(1,p-1)}|\ge 1$.
Obviously,
\EQQ{
|\Phi^{(p)}(\vec{k}^{(p)})|&=|\phi(k_{(1,p)})-\sum_{l=1}^p\phi(k_l)|\\
&\ge |\ga_0|\Big|k_{(1,p)}^5 -\sum_{l=1}^p k_l^5\Big|-C|\ga_1||k_p|^3\\
&\gec |k_p|^4|k_{(1,p-1)}|-C\max_{1\le l\le p-1} |k_l|^5-C|\ga_1||k_p|^3\\
&\gec |k_p|^4|k_{(1,p-1)}|.
}
\end{proof}
\begin{defn}
Let $1\le l <m \le p+1$.
$T(k_l,k_m)$ is the transportation with respect to variables of a multiplier:
\EQQ{
T(k_l,k_m)\big[M(\vec{k}^{(p)})\big]
:=M(k_1,\ldots,k_{l-1},k_m,k_{l+1},\ldots,k_{m-1},k_l,k_{m+1},\ldots,k_{p+1}),
}
and $S(k_l,k_m)$ is the symmetrization operator:
\EQ{
S(k_l,k_m)\big[M(\vec{k}^{(p)})\big]:=\frac{1+T(k_l,k_m)}{2}[M(k_1,\ldots,k_{p+1})\big].\label{sym}
}
\end{defn}
Obviously, $S(k_l,k_m)\big[M(\vec{k}^{(p)})\big]$ is symmetric with $k_l$ and $k_m$, that is to say,
\EQQ{
T(k_l,k_m)\big[S(k_l,k_m)\big[M(\vec{k}^{(p)})\big]\big]=S(k_l,k_m)\big[M(\vec{k}^{(p)})\big].
}
\begin{lem}\label{oscillation_est2}
Let $p\ge 2$, $q \ge 2$ and $|k_{p+q-1}|^{4/5}\sim |k_{p+q}|^{4/5} \ge C\max_{1\le l \le p+q-2}\{|k_l|\}$ for sufficiently large $C=C(p,q)>0$. Then,
\EQ{\label{eq3.96}
&\Big|\big(1-S(k_{p+q-1},k_{p+q})\big)\Big[ \frac{1}{\Phi^{(p)}(k_{1},\ldots,k_{p-1},k_{(p,p+q-1)},k_{p+q})}\Big]\Big|\\
&\gec \max_{1\le l\le p+q-2} \frac{|k_l|^2}{|k_{p+q-1}|^5} \sim \max_{1\le l\le p+q-2} \frac{|k_l|^2}{|k_{p+q}|^5}
}
when $\vec{k}^{(p+q-1)} \in \Z_0^{(p+q-1)}$ and $k_{(1,p-1)}\neq 0$.
\end{lem}
\begin{proof}
The left-hand side of \eqref{eq3.96} is equal to
\EQ{\label{eq3.98}
\Big|\frac{1}{2\Phi^{(p)}}-\frac{1}{2T(k_{p+q-1},k_{p+q})[\Phi^{(p)}]}\Big|
=\Big|\frac{T(k_{p+q-1},k_{p+q})[\Phi^{(p)}]-\Phi^{(p)}}{2\Phi^{(p)} T(k_{p+q-1},k_{p+q})[\Phi^{(p)}]}\Big|.
}
By direct calculation, we have
\EQ{\label{eq3.97}
&|T(k_{p+q-1},k_{p+q})[\Phi^{(p)}(k_{1},\ldots,k_{p-1},k_{(p,p+q-1)},k_{p+q})]\\
& \ \ \ -\Phi^{(p)}(k_{1},\ldots,k_{p-1},k_{(p,p+q-1)},k_{p+q})|\\
\lec& \Big|\Big(-\sum_{l=1}^{p-1}\phi(k_l)-\phi(k_{(p,p+q)}-k_{p+q-1})-\phi(k_{p+q-1})\Big)\\
&-\Big(-\sum_{l=1}^{p-1}\phi(k_l)-\phi(k_{(p,p+q-1)})-\phi(k_{p+q})\Big)\Big|\\
\lec&|-(k_{(p,p+q)}-k_{p+q-1})^5+k_{p+q}^5-k_{p+q-1}^5+k_{(p,p+q-1)}^5|+|k_{p+q}|^3\\
\lec&|k_{p+q}^4-k_{p+q-1}^4||k_{(p,p+q-2)}|+|k_{p+q}|^3\max_{p\le l\le p+q-2}\LR{k_l}^2\\
\lec&|k_{p+q}|^3\max_{1\le l\le p+q-2}|k_l|^2.
}
Combining Lemma \ref{oscillation_est}, \eqref{eq3.98} and \eqref{eq3.97},
we conclude \eqref{eq3.96}. 
\end{proof}
Put
\EQ{\label{menergy1}
F_{s}(f,g):=&\frac{1}{2}\|\p_x^s(f-g)\|_{L^2}^2+\frac{1}{2}\|f-g\|_{L^2}^2\Big(1+C_s\sum_{j=1}^{j_0}\|f\|_{H^4}^{s(p_j-1)}\Big)\\
+& \sum_{j=1}^{j_0}{\Ga}^{(p_j)}\Big( \frac{(ik_{p_j})^{s+1}(ik_{p_j+1})^{s+1}M_{NR,j}}{\Phi^{(p_j)}}; f,\ldots,f,f-g,f-g \Big),
}
where $C_s$ is a sufficiently large constant such that Lemma \ref{sobolev_energy} holds and monotonically increasing with respect to $s$.
We can easily check that $F_s(f,g)$ is real valued because $M_{NR,j}(\vec{k}^{(p_j)})=\overline{M_{NR,j}(-\vec{k}^{(p_j)})}$ and $\Phi^{(p_j)}(\vec{k}^{(p_j)})=\overline{\Phi^{(p_j)}(-\vec{k}^{(p_j)})}$.
Note that
\EQ{\label{EF}
E_s(f)=F_s(f,0).
}
\begin{lem}[comparison lemma]\label{sobolev_energy}
Let $s \in \N$ and $C_s$ be sufficiently large. Then,
for any $f\in H^s(\T)\cap H^4(\T)$ and $g\in H^s(\T)$, it follows that
\EQQ{
F_s(f,g) \le  \|\p_x^s(f-g)\|_{L^2}^2+\|f-g\|_{L^2}^2\Big(1+C_s\sum_{j=1}^{j_0}\|f\|_{H^4}^{s(p_j-1)}\Big) \le 4F_s(f,g).
}
Particularly, by \eqref{EF},
\EQQ{
E_s(f) \le  \|\p_x^s f\|_{L^2}^2+\|f\|_{L^2}^2\Big(1+C_s\sum_{j=1}^{j_0}\|f\|_{H^4}^{s(p_j-1)}\Big) \le 4E_s(f).
}
\end{lem}
\begin{proof}
By the definition \eqref{menergy1}, we only need to show
\EQQ{
I:=&\Big|\Ga^{(p_j)} \Big( \frac{(ik_{p_j})^{s+1}(ik_{p_j+1})^{s+1}M_{NR,j}}{\Phi^{(p_j)}}; f,\ldots,f,f-g,f-g\Big) \Big|\\
\le & \frac{1}{2j_0} \|\p_x^s(f-g)\|_{L^2}^2+\frac{C_s}{2}\|f-g\|_{L^2}^2\|f\|_{H^4}^{s(p_j-1)}.
}
By Lemma \ref{oscillation_est} and \eqref{eq_mnr}, we get
\EQQ{
\Big|\frac{(ik_{p_j})^{s+1}(ik_{p_j+1})^{s+1}M_{NR,j}}{\Phi^{(p_j)}}\Big|\lec |k_{p_j}|^{s-1}|k_{p_j+1}|^{s-1}
\prod_{l=1}^{p_j-1}\LR{k_l}^3.
}
Therefore, by the Sobolev inequality, we obtain
\EQQ{
I \lec \|\p_x^{s-1}(f-g)\|^2_{L^2}\|f\|_{H^4}^{p_j-1}.
}
It follows that $\|\p_x^{s-1}(f-g)\|_{L^2}\lec \|\p_x^s(f-g)\|^{1-1/s}_{L^2}\|f-g\|_{L^2}^{1/s}$ by Lemma \ref{lemGN}. Therefore, we obtain
\EQQ{
I \le C\|\p_x^s(f-g)\|^{2-2/s}_{L^2}\|f-g\|_{L^2}^{2/s}\|f\|_{H^4}^{p_j-1}\le \frac{1}{2j_0}\|\p_x^s(f-g)\|^2_{L^2}+\frac{C_s}{2}\|f-g\|_{L^2}^2\|f\|_{H^4}^{s(p_j-1)}.
}
\end{proof}
\begin{rem}
Though $E_s(f)$ is not monotonically increasing with respect to $s$, by Lemma \ref{sobolev_energy}, we have
\EQ{\label{monotone}
E_{s_1}(f) \le 4E_{s_2}(f)
}
for any $s_1 \le s_2$ and $f\in H^{s_2}(\T)$.
\end{rem}

\begin{prop}[L.W.P. for the regularized equation]\label{para_exist}
Let $\e \in (0,1]$ and $\vp\in H^s(\T)$ with $s \ge 4$. Then, there exist $T_\e \in (0,\infty]$ and the unique solution $u_\e(t) \in C([0,T_\e); H^s(\T))$ to \eqref{es1}--\eqref{es2} on $[0,T_\e)$ such that
(i) $\liminf_{t\to T_\e} \|u_\e(t)\|_{H^4}=\infty$ or (ii) $T_\e=\infty$ holds.
Moreover, we assume $\{\vp^j\}\subset H^s(\T)$ satisfies $\|\vp^j-\vp\|_{H^s}\to 0$ as $j\to \infty$.
Let $u_\e^j(t) \in C([0,T_\e);H^s(\T))$ be the solution to \eqref{es1} with initial data $\vp^j$.
Then, for any $T\in(0,T_\e]$, we have
$\sup_{t\in [0,T]} \|u^j(t)-u(t)\|_{H^s}\to 0$ as $j\to \infty$.
\end{prop}
\begin{proof}[Proof of Proposition \ref{para_exist}]
The result follows from the standard argument by the Banach fixed point theorem.
Therefore, we mention only the outline.
Fix $\e\in (0,1]$. First we consider the case $s=4$.
Put $U_\e(t)$ be the propagator of the linear part of \eqref{es1}, that is to say, $U_\e(t)=\F^{-1}\exp \big(-\e t k^4-it(\ga_0 k^5-\ga_1k^3+\ga_2 k)\big)\F_x$.
Then \eqref{es1}--\eqref{es2} is written into the integral equation:
\EQ{\label{integ}
u_\e=M(u_\e) \quad \text{where} \quad M(u_\e(t))=U_\e(t)\vp+\int_0^t U_\e(t-t')N(u_\e(t'))\, dt'.
}
For sufficiently small $T>0$, we will show that $M(u_\e)$ is a contraction map on
\EQQ{
X:=\big\{u_\e\in C([0,T]; H^4(\T)) : \|u_\e(t)\|_{X} \le 2\|\vp\|_{H^4} \big\},
}
where $\|u_\e\|_{X}:= \sup_{t\in [0,T]}\|u_\e(t)\|_{H^4}$.
First we show $M(u_\e)$ is a map from $X$ into $X$.
Obviously,
\EQQ{
\|M(u_\e(t))\|_{H^4} \le \|\vp\|_{H^4}+\int_0^t \|U_\e(t-t')N(u_\e(t'))\|_{H^4}\, dt'.
}
By Plancherel's identity, we have
\EQQ{
&\int_0^t \|U_\e(t-t')N(u_\e(t'))\|_{H^4}\, dt'\\
= C&\int_0^t \|\LR{k}^3 \exp\big(-\e (t-t') k^4-i(t-t')(\ga_0 k^5-\ga_1k^3+\ga_2 k)\big)\\
&\hspace*{4em} \times \LR{k} \F_x [N(u_\e(t'))](k) \|_{\ell^2_k}\, dt'.
}
By the Sobolev inequality, we obtain
\EQQ{
\|\LR{k} \F_x [N(u_\e(t'))](k) \|_{\ell^2_k}
 \lec \sum_{j=1}^{j_0}\|u_\e(t')\|_{H^4}^{p_j}.
}
Since
\EQQ{
&\sup_{k\in\Z}\LR{k}^3 \big|\exp\big(-\e (t-t') k^4-i(t-t')(\ga_0 k^5-\ga_1k^3+\ga_2 k)\big)\big|\\
 \lec& 1+\e^{-3/4}(t-t')^{-3/4},
}
we conclude
\EQQ{
\sup_{t\in [0,T]}\|M(u_\e(t))\|_{H^4} \le \|\vp\|_{H^4}+C(T+\e^{-3/4}T^{1/4})\sum_{j=1}^{j_0}\|\vp\|_{H^4}^{p_j}<2\|\vp\|_{H^4}
}
for sufficiently small $T=T(\|\vp\|_{H^4},\e)>0$ and any $u_\e\in X$.
By a similar argument, we can easily show that $\|M(u_{1,\e})-M(u_{2,\e})\|_X \le 2^{-1}\|u_{1,\e}-u_{2,\e}\|_{X}$ when $u_{1,\e},u_{2,\e}\in X$. Therefore, $M(u)$ is a contraction map and we obtain the fixed point in $X$, which is the solution to \eqref{integ} on $[0,T]$.
Since $\|u(T)\|_{H^4}$ is finite, we can repeat the argument above with initial data $u(T)$ to obtain the solution on $[T,T+T']$. Iterating this process, we can extend  the solution on $[0,T_\e)$ where $T_\e=\infty$ or $\liminf_{t\to T_\e} \|u(t)\|_{H^4}=\infty$ holds.

Next, we consider the case $s>4$. The solution obtained by the argument above satisfies
\EQ{\label{e25.5}
u_\e=U_\e(t)\vp+\int_0^t U(t-t')N(u_\e(t'))\, dt'.
}
By a similar way as above with Lemma \ref{lemGN}, we obtain
\EQQ{
\sup_{t\in [0,T]}\|u_\e(t)\|_{H^s} \le \|\vp\|_{H^s}+C(s)(T+\e^{-3/4}T^{1/4})\sup_{t\in [0,T]}\|u_\e(t)\|_{H^s}\sum_{j=1}^{j_0}\|\vp\|_{H^4}^{p_j-1}.
}
We take sufficiently small $T=T(\|\vp\|_{H^4},\e,s)>0$ such that
\EQQ{
C(s)(T+\e^{-3/4}T^{1/4})\sum_{j=1}^{j_0}\|\vp\|_{H^4}^{p_j-1}<1/2.
}
Then, we obtain $\sup_{t\in [0,T]}\|u_\e(t)\|_{H^s} <2\|\vp\|_{H^s}$.
By using \eqref{e25.5}, we also obtain $u_\e\in C([0,T];H^s(\T))$.
Since $\|u(T)\|_{H^s}$ is finite, we can repeat the argument above with initial data $u(T)$ to obtain $u_\e\in C([T,T+T'];H^s(\T))$.
We can iterate this process as far as $\|u_\e(t)\|_{H^4}< \infty$. Therefore, we conclude $u_\e \in C([0,T_\e);H^s(\T))$.
We omit the proof of the uniqueness and the continuous dependence because it follows from the standard argument.
\end{proof}


\section{refined energy estimate and existence of solution}
This section is devoted to show the following proposition, which is the main estimate in this paper.
\begin{prop}[energy inequality for the difference of two solutions]\label{thm_energy_est}
Let $s\in \N$, $s\ge 8$, $\e_1, \e_2\in [0,1]$ and $u\in L^\infty([0,T];H^s(\T))$ (resp. $v \in L^\infty([0,T];H^{s+4}(\T))$) be a solution to \eqref{es1} with $\e=\e_1$ (resp. $\e=\e_2$) on $[0,T]$.
Then, it follows that
\EQQ{
&\frac{d}{dt} F_{s}(u(t),v(t))
+P_N(u(t))\|\p_x^{s+1} (u(t)-v(t))\|_{L^2}^2\\
\lec&  F_s(u(t),v(t))(1+E_{8}(u(t))+E_{8}(v(t)))^{r(s)}\\
+&\big(F_8(u(t),v(t))(1+E_8(u(t))+E_8(v(t)))^{p_{\max}-2}
+|\e_1-\e_2|^2\big)\\
&\times\big(E_s(u(t))+E_{s+4}(v(t))\big),
}
on $[0,T]$, where $r(s):=s(p_{\max}-1)$ and the implicit constant depends on $s$ and does not depend on $u, v, \e_1,\e_2$, and $T$.
\end{prop}
Before we proceed to the proof of Proposition \ref{thm_energy_est}, we prepare some lemmas and propositions.
\begin{lem}\label{lem_ga1}
Let $s \in \N\cup\{0\}$. Then, for any $f\in H^8(\T)$ and $g\in H^s(\T)$, it follows that
\EQ{\label{eq_ga1}
&\big|{\Ga}^{(p_j)}\big(\la_j a_j (ik_{p_j})^{s+2}(ik_{p_j+1})^{s+1}D_{a_j-1,b_j,c_j};f,\ldots,f,g,g\big)\\
&-{\Ga}^{(p_j)}\big( M_{1,*};f,\ldots,f,g,g\big)\big|\\
\lec& \|f\|_{H^4}^{p_j-2}\|f\|_{H^{8}}\|g\|_{H^s}^2,
}
where $M_{1,*}:=-\frac{1}{2}\la_ja_j(ik_{(1,p_j-1)})(ik_{p_j})^{s+1}(ik_{p_j+1})^{s+1}D_{a_j-1,b_j,c_j}M_{H}^{(p_j)}M_{NZ}^{(p_j)}$.
\end{lem}
\begin{proof}
By Lemma \ref{byparts_lem}, we have
\EQQ{
&{\Ga}^{(p_j)}\big( \la_ja_j(ik_{p_j})^{s+2}(ik_{p_j+1})^{s+1}D_{a_j-1,b_j,c_j};f,\ldots,f,g,g\big)\\
=&-\frac{1}{2}{\Ga}^{(p_j)}\big( \la_ja_j(ik_{(1,p_j-1)})(ik_{p_j})^{s+1}(ik_{p_j+1})^{s+1}D_{a_j-1,b_j,c_j}M_{NZ}^{(p_j)};f,\ldots,f,g,g\big).
}
Therefore, the left-hand side of \eqref{eq_ga1} is bounded by
\EQ{\label{eq3.3.1}
\frac{1}{2}\big|{\Ga}^{(p_j)}\big( \la_ja_j(ik_{(1,p_j-1)})(ik_{p_j})^{s+1}(ik_{p_j+1})^{s+1}D_{a_j-1,b_j,c_j}(1-M_{H}^{(p_j)})M_{NZ}^{(p_j)};f,\ldots,f,g,g\big)\big|
}
By Lemma \ref{lem_MH}, we have
\EQQ{
&\big|(ik_{(1,p_j-1)})(ik_{p_j})^{s+1}(ik_{p_j+1})^{s+1}D_{a_j-1,b_j,c_j}(1-M_{H}^{(p_j)})M_{NZ}^{(p_j)}\big|\\
\lec &|k_{p_j}|^s|k_{p_j+1}|^s\max_{1\le l\le p_j-1} \{|k_l|^{7/2}\}\prod_{l=1}^{p_j-1}\LR{k_l}^3.
}
Therefore, \eqref{eq3.3.1} is bounded by $\|f\|_{H^4}^{p_j-2}\|f\|_{H^{8}}\|g\|_{H^s}^2$ by the Sobolev inequality.
\end{proof}
\begin{lem}\label{lem_ga23}
Let $s \in \N\cup\{0\}$. Put
\EQQS{
&M_{2,*}:=\la_jb_j(ik_{p_j})^{s+1}(ik_{p+1})^{s+1}D_{a_j,b_j-1,c_j}M_{H}^{(p_j)}M_{NZ}^{(p_j)},\\
&M_{4,*}:=\la_ja_j(ik_{(1,p_j-1)})(ik_{p_j})^{s+1}(ik_{p+1})^{s+1}D_{a_j-1,b_j,c_j}M_{H}^{(p_j)}M_{NZ}^{(p_j)}.
}
Then, it follows that
\EQ{\label{eq_ga2}
&\big|{\Ga}^{(p_j)}\big( \la_jb_j(ik_{p_j})^{s+1}(ik_{p_j+1})^{s+1}D_{a_j,b_j-1,c_j};f,\ldots,f,g,g\big)-{P}_{N_j}(f)\|\p_x^{s+1} g\|_{L^2}^2\\
&-{\Ga}^{(p_j)}\big(M_{2,*} ;f,\ldots,f,g,g\big)\big|\\
\lec& \|f\|_{H^4}^{p_j-2}\|f\|_{H^{7}}\|g\|_{H^s}^2
}
for any $f\in H^{7}(\T)$ and $g\in H^s(\T)$, and it follows that
\EQ{\label{eq_ga3}
&\big|{\Ga}^{(p_j)}\big( \la_ja_j(ik_{(1,p_j-1)})(ik_{p_j})^{s+1}(ik_{p_j+1})^{s+1}D_{a_j-1,b_j,c_j};f,\ldots,f,g,g\big)\\
&-{\Ga}^{(p_j)}\big(M_{4,*};f,\ldots,f,g,g\big)\big|\\
\lec& \|f\|_{H^4}^{p_j-2}\|f\|_{H^8}\|g\|_{H^s}^2
}
for any $f\in H^{8}(\T)$ and $g\in H^s(\T)$.
\end{lem}
\begin{proof}
Since
\EQQ{
{P}_{N_j}(f)\|\p_x^{s+1} g\|_{L^2}^2={\Ga}^{(p_j)}\big( \la_jb_j(ik_{p_j})^{s+1}(ik_{p_j+1})^{s+1}D_{a_j,b_j-1,c_j}(1-M_{NZ}^{(p_j)});f,\ldots,f,g,g\big)
}
and
\EQQ{
1-(1-M_{NZ}^{(p_j)})-M_{H}^{(p_j)}M_{NZ}^{(p_j)}=(1-M_{H}^{(p_j)})M_{NZ}^{(p_j)},
}
the left-hand side of \eqref{eq_ga2} is equal to
\EQ{\label{e3.3.2}
\big|{\Ga}^{(p_j)}\big( \la_jb_j(ik_{p_j})^{s+1}(ik_{p_j+1})^{s+1}D_{a_j,b_j-1,c_j}(1-M_{H}^{(p_j)})M_{NZ}^{(p_j)};f,\ldots,f,g,g\big)\big|
} 
By Lemma \ref{lem_MH}, we have
\EQQ{
&|(ik_{p_j})^{s+1}(ik_{p_j+1})^{s+1}D_{a_j,b_j-1,c_j}(1-M_{H}^{(p_j)})M_{NZ}^{(p_j)}|\\
\lec &|k_{p_j}|^s|k_{p_j+1}|^s\max_{1\le l\le p_j-1} \{|k_l|^{5/2}\}\prod_{l=1}^{p_j-1}\LR{k_l}^3.
}
Therefore, \eqref{e3.3.2} is bounded by $\|f\|_{H^4}^{p_j-2}\|f\|_{H^{7}}\|g\|_{H^s}^2$ by the Sobolev inequality.
In the same manner, by Lemma \ref{lem_MH}, we have
\EQQ{
&\big|(ik_{(1,p_j-1)})(ik_{p_j})^{s+1}(ik_{p+1})^{s+1}D_{a_j-1,b_j,c_j}(1-M_{H}M_{NZ})\big|\\
=&\big|(ik_{(1,p_j-1)})(ik_{p_j})^{s+1}(ik_{p+1})^{s+1}D_{a_j-1,b_j,c_j}(1-M_{H})M_{NZ}\big|\\
\lec& |k_{p_j}|^s|k_{p_j+1}|^s\max_{1\le l\le p_j-1} \{|k_l|^{7/2}\}\prod_{l=1}^{p_j-1}\LR{k_l}^3.
}
Therefore, by the Sobolev inequality, we get \eqref{eq_ga3} .
\end{proof}
\begin{lem}\label{Ki}
Let $s \in \N, s\ge 7$ and
put
\EQQ{
K_i:={\Ga}^{(p_j)}\Big(\frac{(ik_{p_j+1})^{s+1}M_{NR,j}}{\Phi^{(p_j)}}; f,\ldots,f,\p_x^{s+1}\big(N_i(f)-N_i(g)\big),f-g\Big).
}
Then, for any $f\in H^s(\T)\cap H^8(\T)$ and $g\in H^{s+4}(\T)$, we have
\EQQ{
|K_i| \lec& \|f-g\|_{H^s}^2(\|f\|_{H^8}+\|g\|_{H^8})^{p_i+p_j-2}\\
+&\|f-g\|_{H^s}\|f-g\|_{H^8}(\|f\|_{H^s}+\|g\|_{H^s})(\|f\|_{H^8}+\|g\|_{H^8})^{p_i+p_j-3}\\
+&\|f-g\|_{H^{s-3}}\|f-g\|_{H^4}\|g\|_{H^{s+4}}(\|f\|_{H^4}+\|g\|_{H^4})^{p_i+p_j-3}.
}
\end{lem}
\begin{proof}
Since
\EQQ{
\p_x^{s+1}(N_i(f)-N_i(g))
=&\la_ia_i(\p_x^3 f)^{a_i-1}(\p_x^2 f)^{b_i}(\p_x f)^{c_i}(f)^{d_i}\big(\p_x^{s+4}(f-g)\big)\\
+&\la_ia_i(\p_x^3 f)^{a_i-1}(\p_x^2 f)^{b_i}(\p_x f)^{c_i}(f)^{d_i}(\p_x^{s+4}g)\\
-&\la_ia_i(\p_x^3 g)^{a_i-1}(\p_x^2 g)^{b_i}(\p_x g)^{c_i}(g)^{d_i}(\p_x^{s+4}g)\\
+&\p_x^{s+1}N_i(f)-\la_ia_i(\p_x^{s+4}f)(\p_x^3 f)^{a_i-1}(\p_x^2 f)^{b_i}(\p_x f)^{c_i}(f)^{d_i}\\
-&\p_x^{s+1}N_i(g)+\la_ia_i(\p_x^{s+4}g)(\p_x^3 g)^{a_i-1}(\p_x^2 g)^{b_i}(\p_x g)^{c_i}(g)^{d_i},
}
we have
\EQQ{
K_i=&\Ga^{(p_i+p_j-1)}\big((ik_{p_i+p_j-1})^3M_{i,1} ;f,\ldots,f,f-g,f-g\Big)\\
+&\Big(\Ga^{(p_i+p_j-1)}\big((ik_{p_i+p_j-1})^3M_{i,1} ;f,\ldots,f,g,f-g\big)\\
&-\Ga^{(p_i+p_j-1)}\big((ik_{p_i+p_j-1})^3M_{i,1} ;f,\ldots,f,g,\ldots,g,g,f-g\big)\Big)\\
+&\Big(\Ga^{(p_i+p_j-1)}\big(M_{H,*}M_{i,2} ;f,\ldots,f,f-g\big)\\
&-\Ga^{(p_i+p_j-1)}\big(M_{H,*}M_{i,2} ;f,\ldots,f,g,\ldots,g,f-g\big)\Big)\\
+&\Big(\Ga^{(p_i+p_j-1)}\big((1-M_{H,*})M_{i,2} ;f,\ldots,f,f-g\big)\\
&-\Ga^{(p_i+p_j-1)}\big((1-M_{H,*})M_{i,2} ;f,\ldots,f,g,\ldots,g,f-g\big)\Big)\\
=:&K_{i,1}+K_{i,2}+K_{i,3}+K_{i,4}
}
where the multipliers on $\Z_0^{(p_i+p_j-1)}$ are defined by
\EQQS{
&M_{i,1}:=\frac{\la_ia_i(ik_{p_i+p_j-1})^{s+1}(ik_{p_i+p_j})^{s+1}\ti{M}_{NR,i}\ti{D}_{a_i-1,b_i,c_i}}{\ti{\Phi}_i^{(p_i+p_j-1)}},\\
&M_{i,2}:=\frac{\la_i(ik_{p_i+p_j})^{s+1}\ti{M}_{NR,i}\ti{D}_{a_i,b_i,c_i}}{\ti{\Phi}_i^{(p_i+p_j-1)}}\times \Big((ik_{(p_j,p_i+p_j-1)})^{s+1}-\sum_{l=p_j}^{p_j+a_i-1} (ik_l)^{s+1}\Big),\\
&\ti{\Phi}_i^{(p_i+p_j-1)}:=\Phi^{(p_j)}(k_1,\ldots,k_{p_j-1},k_{(p_j,p_i+p_j-1)},k_{p_i+p_j}),\\
&\ti{M}_{NR,i}:=M_{NR,j}(k_1,\ldots,k_{p_j-1},k_{(p_j,p_i+p_j-1)},k_{p_i+p_j}),\\
&\ti{D}_{a_i,b_i,c_i}:=\prod_{l=p_j}^{p_j+a_i-1}(ik_l)^3\prod_{l=p_j+a_i}^{p_j+a_i+b_i-1}(ik_l)^2\prod_{l=p_j+a_i+b_i}^{p_j+a_i+b_i+c_i-1}(ik_l)\\
&\hspace*{3.3em}=D_{a_i,b_i,c_i}(k_{p_j},\ldots,k_{p_j+p_i-1}),\\
&M_{H,*}:=
\begin{cases}
1, \,\, \text{when}\,\, |k_{(p_j,p_i+p_j-1)}| \ll \max_{p_j\le l\le p_i+p_j-1} |k_l|,\\
0, \,\, \text{otherwise}.
\end{cases}
}
Note that
\EQS{
&|\ti{D}_{a_i-1,b_i,c_i}|\lec \prod_{l=p_j}^{p_i+p_j-2}\LR{k_l}^3,\label{ee3.50}\\
&\big|\ti{D}_{a_i,b_i,c_i} \big((ik_{(p_j,p_i+p_j-1)})^{s+1}-\sum_{l=p_j}^{p_j+a_i-1} (ik_l)^{s+1}\big)\big| \lec\max_{p_j\le l\le p_i+p_j-1} |k_l|^{s-1}\prod_{l=p_j}^{p_i+p_j-1} \LR{k_l}^4.\label{ee3.501}
}
By Lemma \ref{oscillation_est} and \eqref{eq_mnr},
\EQ{\label{ee3.51}
\big|\ti{M}_{NR,i}/\ti{\Phi}_i^{(p_i+p_j-1)}\big| \lec |k_{p_i+p_j}|^{-4}\prod_{l=1}^{p_j-1}\LR{k_l}^3
\sim |k_{(p_j,p_i+p_j-1)}|^{-4}\prod_{l=1}^{p_j-1}\LR{k_l}^3.
}
By \eqref{ee3.50} and \eqref{ee3.51},
\EQQ{
|(ik_{p_i+p_j-1})^3M_{i,1}|\lec |k_{p_i+p_j-1}|^{s+4}|k_{p_i+p_j}|^{s-3}
\prod_{l=1}^{p_i+p_j-2}\LR{k_l}^3.
}
Thus, by $(ii)$ of Lemma \ref{lem_ga4} with $p=p_i+p_j-1, m=p_j, d=s+4, c=s-3, b=0, a=3$, we obtain
\EQQ{
|K_{i,2}|\lec \|f-g\|_{H^{s-3}}\|g\|_{H^{s+4}}\|f-g\|_{H^4}\big(\|f\|_{H^4}+\|g\|_{H^4} \big)^{p_i+p_j-3}.
}
When $\vec{k}^{(p_i+p_j-1)}\in \supp M_{H,*}$,
there exist $l_1, l_2$ such that $p_j\le l_1 <l_2\le p_i+p_j-1$ and
\EQ{\label{ee3.52}
|k_{l_1}|\sim|k_{l_2}|\sim 
\max_{p_j\le l\le p_i+p_j-1} |k_l|.
}
By \eqref{ee3.501}, \eqref{ee3.51} and \eqref{ee3.52}, we get
\EQQ{
|M_{H,*}M_{i,2}|\lec |k_{p_i+p_j}|^{s-3}\LR{k_{l_1}}^{s-4}\LR{k_{l_2}}^3
\prod_{1\le l\le p_i+p_j-1}\LR{k_l}^4.
}
Therefore, by $(i)$ of Lemma \ref{lem_ga4} with $p=p_i+p_j-1, m=p_j,
c=s-3, b=s-7, a=7$, we obtain
\EQQ{
|K_{i,3}|\lec &\|f-g\|_{H^{s-3}}\big(\|f\|_{H^8}+\|g\|_{H^8} \big)^{p_i+p_j-3}\\
\times&\big(\|f-g\|_{H^s}(\|f\|_{H^8}+\|g\|_{H^8} )+\|f-g\|_{H^8}(\|f\|_{H^s}+\|g\|_{H^s}) \big).
}
By \eqref{ee3.501} and \eqref{ee3.51},
\EQQ{
|(1-M_{H,*})M_{i,2}|\lec |k_{p_i+p_j}|^{s}\max_{p_j\le l\le p_i+p_j-1}|k_l|^{s-4}\prod_{l=1}^{p_i+p_j-1}\LR{k_l}^4.
}
Thus, by Lemma \ref{lem_ga4} $(i)$ with $p=p_i+p_j-1, m=p_j, c=s, b=s-4, a=4$, it follows that
\EQQ{
|K_{i,4}|\lec& \|f-g\|_{H^s}(\|f\|_{H^5}+\|g\|_{H^5})^{p_i+p_j-3}\\
&\times\{ \|f-g\|_{H^s}(\|f\|_{H^5}+\|g\|_{H^5})+\|f-g\|_{H^5}(\|f\|_{H^s}+\|g\|_{H^s})\}.
}
Finally, we estimate $K_{i,1}$. It still has one derivative loss because
$(ik_{p_i+p_j-1})^3M_{i,1}$ includes $(ik_{p_i+p_j-1})^{s+4}(ik_{p_i+p_j})^{s+1}$ and \eqref{ee3.51} is not enough to cancel $|k_{p_i+p_j-1}|^{4}|k_{p_i+p_j}|$.
We would like to recover it by symmetry and Lemma \ref{byparts_lem}.
However, $M_{i,1}$ is not symmetric with respect to $k_{p_i+p_j-1}$ and $k_{p_i+p_j}$
because of $\ti{\Phi}_i^{(p_i+p_j-1)}$ and $\ti{M}_{NR,i}$.
To avoid this difficulty, we introduce
\EQQ{
\ti{M}_{NR,i}^{sym}&:=
T(k_{p_i+p_j-1},k_{p_i+p_j})\big[M_{H}^{(p_j)}(k_1,\ldots,k_{p_j-1},k_{(p_j,p_i+p_j-1)},k_{p_i+p_j})\big]\ti{M}_{NR,i}\\
&=M_{H}^{(p_j)}(k_1,\ldots,k_{p_j-1},k_{(p_j,p_i+p_j-2)}+k_{p_i+p_j},k_{p_i+p_j-1})\ti{M}_{NR,i}.
}
and decompose $M_{i,1}$ into four parts:
\EQQ{
M_{i,1}=&\la_i a_i(ik_{p_i+p_j-1})^{s+1}(ik_{p_i+p_j})^{s+1}\ti{D}_{a_i-1,b_i,c_i}\Big( (\ti{M}_{NR,i}-\ti{M}_{NR,i}^{sym})/\ti{\Phi}_i^{(p_i+p_j-1)}\\
&+(1-S(k_{p_i+p_j-1},k_{p_i+p_j}))\big[\ti{M}_{NR,i}^{sym}/\ti{\Phi}_i^{(p_i+p_j-1)}\big]\\
&+S(k_{p_i+p_j-1},k_{p_i+p_j})\big[M_{H,*}\ti{M}_{NR,i}^{sym}/\ti{\Phi}_i^{(p_i+p_j-1)}\big]\\
&+S(k_{p_i+p_j-1},k_{p_i+p_j})\big[(1-M_{H,*})\ti{M}_{NR,i}^{sym}/\ti{\Phi}_i^{(p_i+p_j-1)}\big]\Big)\\
=:&M_{i,11}+M_{i,12}+M_{i,13}+M_{i,14}.
}
Note that $\ti{M}_{NR,i}^{sym}$ is symmetric with respect to $k_{p_i+p_j-1}$ and $k_{p_i+p_j}$.
Recall that $S(k_{p_i+p_j-1},k_{p_i+p_j})$ is the symmetrization operator defined by \eqref{sym}.
Thus, $M_{i,14}$ is symmetric with respect to $k_{p_i+p_j-1}$ and $k_{p_i+p_j}$.
By Lemma \ref{byparts_lem}, we have
\EQ{\label{M14a}
&\Ga^{(p_i+p_j-1)}\big((ik_{p_i+p_j-1})^3M_{i,14} ;f,\ldots,f,f-g,f-g\Big)\\
=&-\frac{1}{2}\Ga^{(p_i+p_j-1)}\big((ik_{(1,p_i+p_j-2)})^3M_{i,14} ;f,\ldots,f,f-g,f-g\Big)\\
+&\frac{3}{2}\Ga^{(p_i+p_j-1)}\big((ik_{(1,p_i+p_j-2)})(ik_{p_i+p_j-1})(ik_{p_i+p_j})M_{i,14} ;f,\ldots,f,f-g,f-g\Big).
}
It follows that $|k_{p_i+p_j-1}| \lec |k_{(p_j,p_i+p_j-1)}|$ when $\vec{k}^{(p_i+p_j-1)}\in \supp 1-M_{H,*}$.
Thus, by \eqref{ee3.51},
\EQQ{
|S(k_{p_i+p_j-1},k_{p_i+p_j})\big[(1-M_{H,*})\ti{M}_{NR,i}^{sym}/\ti{\Phi}_i^{(p_i+p_j-1)}\big]|
\lec |k_{p_i+p_j-1}|^{-2}|k_{p_i+p_j}|^{-2}\prod_{l=1}^{p_j-1}\LR{k_l}^3.
}
Therefore, we obtain
\EQQ{
&|(ik_{(1,p_i+p_j-2)})(ik_{p_i+p_j-1})(ik_{p_i+p_j})M_{i,14}|\\
\lec &|k_{p_i+p_j-1}|^{s}|k_{p_i+p_j}|^{s}\max_{1\le l \le p_i+p_j-2} |k_l| \prod_{l=1}^{p_i+p_j-2} \LR{k_l}^3.
}
In the same manner,
\EQQ{
|(ik_{(1,p_i+p_j-2)})^3M_{i,14}|
\lec |k_{p_i+p_j-1}|^{s-1}|k_{p_i+p_j}|^{s-1}\max_{1\le l\le p_i+p_j-2} |k_l|^{3}\prod_{l=1}^{p_i+p_j-2} \LR{k_l}^3.
}
Thus, applying the Sobolev inequality for \eqref{M14a}, we conclude
\EQ{\label{M14b}
&\Big|\Ga^{(p_i+p_j-1)}\big((ik_{p_i+p_j-1})^3M_{i,14} ;f,\ldots,f,f-g,f-g\Big)\Big|\\
\lec &\|f-g\|^2_{H^s}\|f\|_{H^7}^{p_i+p_j-2}.
}
By \eqref{ee3.51} and \eqref{ee3.52},
\EQQ{
\big|M_{H,*}\ti{M}_{NR,i}^{sym}/\ti{\Phi}_i^{(p_i+p_j-1)}\big|
\lec |k_{l_1}|^4|k_{p_i+p_j-1}|^{-4}|k_{p_i+p_j}|^{-4}\prod_{l=1}^{p_j-1} \LR{k_l}^3,
}
where $p_j\le l_1 \le p_i+p_j-2$. Thus, by \eqref{ee3.50},
\EQQ{
|(ik_{p_i+p_j-1})^3M_{i,13}|
\lec |k_{p_i+p_j-1}|^{s}|k_{p_i+p_j}|^{s-3}|k_{l_1}|^4\prod_{l=1}^{p_i+p_j-2} \LR{k_l}^3.
}
Thus, by the Sobolev inequality, we obtain
\EQ{\label{M13}
&\Big|\Ga^{(p_i+p_j-1)}\big((ik_{p_i+p_j-1})^3M_{i,13} ;f,\ldots,f,f-g,f-g\Big)\Big|\\
\lec & \|f-g\|^2_{H^s}\|f\|_{H^8}^{p_i+p_j-2}.
}
By Lemma \ref{oscillation_est2} and \eqref{eq_mnr},
\EQQ{
&\big|(1-S(k_{p_i+p_j-1},k_{p_i+p_j}))
\big[\ti{M}_{NR,i}^{sym}/\ti{\Phi}_i^{(p_i+p_j-1)}\big]\big|\\
=&\big|\ti{M}_{NR,i}^{sym}(1-S(k_{p_i+p_j-1},k_{p_i+p_j}))
\big[1/\ti{\Phi}_i^{(p_i+p_j-1)}\big]\big|\\
\lec & |k_{(1,p_j-1)}|\prod_{l=1}^{p_j-1}\LR{k_l}^3 |k_{p_i+p_j-1}|^{-4} |k_{p_i+p_j}|^{-1}\max_{1\le l\le p_i+p_j-2} |k_l|^2.
}
Thus, by \eqref{ee3.50}
\EQQ{
|(ik_{p_i+p_j-1})^3M_{i,12}|
\lec |k_{p_i+p_j-1}|^{s}|k_{p_i+p_j}|^{s}\max_{1\le l\le p_i+p_j-2} |k_l|^3\prod_{l=1}^{p_i+p_j-2} \LR{k_l}^3.
}
Therefore, by the Sobolev inequality, we obtain
\EQ{\label{M12}
&\Big|\Ga^{(p_i+p_j-1)}\big((ik_{p_i+p_j-1})^3M_{i,12} ;f,\ldots,f,f-g,f-g\Big)\Big|\\
\lec & \|f-g\|^2_{H^s}\|f\|_{H^7}^{p_i+p_j-2}.
}
Since
\EQQ{
\ti{M}_{NR,i}-\ti{M}_{NR,i}^{sym}=(1-M_{H}^{(p_j)}\big(k_1,\ldots,k_{p_j-1},k_{(p_j,p_i+p_j-2)}+k_{p_i+p_j},k_{p_i+p_j-1})\big)\ti{M}_{NR,i},
}
there exists $l_3$ such that $1\le l_3\le p_j-1$, $|k_{l_3}|^5 \gec |k_{p_i+p_j-1}|^4$ if $\vec{k}^{(p_i+p_j-1)}\in \supp \big(\ti{M}_{NR,i}-\ti{M}_{NR,i}^{sym}\Big)$.
Thus, by \eqref{ee3.50} and \eqref{ee3.51}
\EQQ{
|(ik_{p_i+p_j-1})^3M_{i,11}|
\lec |k_{l_3}|^5|k_{p_i+p_j-1}|^{s}|k_{p_i+p_j}|^{s-3}\prod_{l=1}^{ p_i+p_j-2} \LR{k_l}^3.
}
By the Sobolev inequality, we obtain
\EQ{\label{M11}
&\Big|\Ga^{(p_i+p_j-1)}\big((ik_{p_i+p_j-1})^3M_{i,11} ;f,\ldots,f,f-g,f-g\Big)\Big|\\
\lec & \|f-g\|^2_{H^s}\|f\|_{H^8}^{p_i+p_j-2}.
}
Collecting \eqref{M14b}--\eqref{M11}, we conclude
$|K_{i,1}|\lec \|f-g\|_{H^s}^2\|f\|_{H^8}^{p_i+p_j-2}$.
\end{proof}

\begin{prop}\label{multi_est}
Let $s\in \N, s\ge 7$.
Then, it follows that
\EQQ{
&\big| (-1)^s\big(N_j(f)-N_j(g), \p_x^{2s}(f-g)\big)_{L^2}
+P_{N_j}(f)\|\p_x^{s+1} (f-g)\|_{L^2}^2\\
&+\Ga^{(p_j)}((ik_{p_j})^{s+1}(ik_{p_j+1})^{s+1}M_{NR,j};f,\ldots,f,f-g,f-g)\big|\\
\lec& \|f-g\|^2_{H^s}(\|f\|_{H^8}+\|g\|_{H^8})^{p_j-1}\\
&+\|f-g\|^2_{H^7}(\|f\|_{H^s}+\|g\|_{H^s})^2(\|f\|_{H^7}+\|g\|_{H^7})^{p_j-3}\\
&+\|f-g\|^2_{H^7}\|g\|^2_{H^{s+3}}(\|f\|_{H^7}+\|g\|_{H^7})^{p_j-3}\\
}
for any $f\in H^s(\T) \cap H^{8}(\T)$ and $g\in H^{s+3}(\T)$.
\end{prop}
\begin{proof}
By \eqref{eq2.1},
\EQ{\label{E3.1}
&(-1)^s\big(({N}_j(f)-N_j(g)), \p_x^{2s}(f-g)\big)_{L^2}\\
=&(-1)^s{\Ga}^{(p_j)}\big(\la_j(ik_{p_j+1})^{2s}D_{a_j,b_j,c_j};f,\ldots,f,f-g \big)\\
-&(-1)^s{\Ga}^{(p_j)}\big(\la_j(ik_{p_j+1})^{2s}D_{a_j,b_j,c_j};g,\ldots,g,f-g \big).
}
For $\vec{k}^{(p_j)}\in \Z_0^{(p_j)}$, it follows that
\EQ{\label{E3.2}
(-1)^s(ik_{p_j+1})^{s-1}=-(ik_{(1,p_j)})^{s-1}
=-\sum_{n=1}^{7} S_n,
}
where
\EQQ{
&S_1:=\sum_{l=1}^{a_j} (ik_l)^{s-1},\,\,
S_2:=\sum_{l=a_j+1}^{a_j+b_j} (ik_l)^{s-1},\,\,
S_3:=\sum_{l=a_j+b_j+1}^{a_j+b_j+c_j} (ik_l)^{s-1},\\
&S_4:=(s-1)\sum_{l=1}^{a_j} (ik_l)^{s-2}(ik_{(1,p_j)}-ik_l),\,\,\, S_5:=(s-1)\sum_{l=a_j+1}^{a_j+b_j} (ik_l)^{s-2}(ik_{(1,p_j)}-ik_l),\\
&S_6:=\frac{(s-1)(s-2)}{2}\sum_{l=1}^{a_j} (ik_l)^{s-3}(ik_{(1,p_j)}-ik_l)^2,\,\, S_7:=(ik_{(1,p_j)})^{s-1}-\sum_{n=1}^{6} S_n.
}
From \eqref{E3.1} and \eqref{E3.2}, we have
\EQ{\label{E3.3}
(-1)^s\big(N_j(f)-N_j(g), \p_x^{2s}(f-g)\big)_{L^2}=-\sum_{n=1}^{7} \big(I_n(f)-I_n(g)\big)
}
where
\EQQ{
I_n(h):={\Ga}^{(p_j)}\big( \la_j(ik_{p_j+1})^{s+1} D_{a_j,b_j,c_j}S_n;h,\ldots,h,f-g\big).
}
Since
\EQQ{
|S_7|\lec \max_{1\le l\le p_j} |k_l|^{s-7}\prod_{l=1}^{a_j}\LR{k_l}^3\prod_{l=a_j+1}^{a_j+b_j}\LR{k_l}^4
\prod_{l=a_j+b_j+1}^{a_j+b_j+c_j}\LR{k_l}^5\prod_{l=a_j+b_j+c_j}^{p_j}\LR{k_l}^6,
}
and $|k_{p_j+1}|=|k_{(1,p_j)}|\lec \max_{1\le l\le p_j} |k_l|$,
we have
\EQQ{
|(ik_{p_j+1})^{s+1} D_{a_j,b_j,c_j}S_7|\lec |k_{p_j+1}|^s\max_{1\le l \le p_j}|k_l|^{s-6}\prod_{l=1}^{p_j}\LR{k_l}^6.
}
Therefore, by $(ii)$ of Lemma \ref{lem_ga4} with $p=p_j, m=1, c=s, b=s-6, a=6$,
we have
\EQ{\label{E3.4}
|I_7(f)-I_7(g)|
&\lec \|f-g\|_{H^s}(\|f\|_{H^7}+\|g\|_{H^7})^{p_j-2}\\
\times&\{\|f-g\|_{H^s}(\|f\|_{H^7}+\|g\|_{H^7})+\|f-g\|_{H^7}
(\|f\|_{H^s}+\|g\|_{H^s})\}\\
&\lec \|f-g\|_{H^s}^2(\|f\|_{H^7}+\|g\|_{H^7})^{p_j-1}\\
+ &\|f-g\|^2_{H^7}(\|f\|_{H^s}
+ \|g\|_{H^s})^2(\|f\|_{H^7}+\|g\|_{H^7})^{p_j-3}.
}
Since $D_{a_j,b_j,c_j}S_1$ is symmetric with respect to $k_1,\ldots,k_{a_j}$, we have
\EQQ{
I_1(h)={\Ga}^{(p_j)}\big( \la_ja_j(ik_1)^{s-1}(ik_{p_j+1})^{s+1}D_{a_j,b_j,c_j};h,\ldots,h,f-g\big).
}
Changing the role of $k_1$ and $k_{p_j}$, we have
\EQQ{
I_1(h)={\Ga}^{(p_j)}\big(M_1 ;h,\ldots,h,f-g\big), \qquad
M_1:=\la_ja_j(ik_{p_j})^{s+2}(ik_{p_j+1})^{s+1} D_{a_j-1,b_j,c_j}.
}
Note that $D_{a_j,b_j,c_j}S_2$ and $D_{a_j,b_j,c_j}S_5$ are symmetric with respect to $k_{a_j+1},\ldots,k_{a_j+b_j}$,
$D_{a_j,b_j,c_j}S_4$ and $D_{a_j,b_j,c_j}S_6$ are symmetric with respect to $k_{1},\ldots,k_{a_j}$ and
$D_{a_j,b_j,c_j}S_3$ is symmetric with respect to $k_{a_j+b_j+1},\ldots,k_{a_j+b_j+c_j}$.
Therefore, in the same manner as $I_1(h)$, we obtain
\EQQ{
I_n(h)={\Ga}^{(p_j)}\big(M_n ;h,\ldots,h,f-g\big),
}
for $2\le n\le 6$, where
\EQQ{
&M_2:=\la_jb_j(ik_{p_j})^{s+1}(ik_{p_j+1})^{s+1} D_{a_j,b_j-1,c_j}\\
&M_3:=\la_jc_j(ik_{p_j})^{s}(ik_{p_j+1})^{s+1} D_{a_j,b_j,c_j-1},\\
&M_4:=(s-1)\la_ja_j(ik_{(1,p_j-1)})(ik_{p_j})^{s+1}(ik_{p_j+1})^{s+1} D_{a_j-1,b_j,c_j},\\
&M_5:=(s-1)\la_jb_j(ik_{(1,p_j-1)})(ik_{p_j})^{s}(ik_{p_j+1})^{s+1} D_{a_j,b_j-1,c_j},\\
&M_6:=(s-1)(s-2)\la_ja_j(ik_{(1,p_j-1)})^2(ik_{p_j})^{s}(ik_{p_j+1})^{s+1} D_{a_j-1,b_j,c_j}/2.
}
Moreover, for $1\le n\le 6$, we have
\EQ{\label{E3.5}
I_n(f)-I_n(g)&={\Ga}^{(p_j)}\big(M_n ;f,\ldots,f,f-g,f-g\big)\\
&+{\Ga}^{(p_j)}\big(M_n ;f,\ldots,f,g,f-g\big)-{\Ga}^{(p_j)}\big(M_n ;g,\ldots,g,g,f-g\big).
}
Since
\EQQ{
|k_{p_j+1}|=|k_{(1,p_j)}|\lec \LR{k_{p_j}}\max_{1\le l\le p_j-1} |k_l|,
}
it follows that
\EQQ{
|M_n| \lec \LR{k_{p_j+1}}^{s}\LR{k_{p_j}}^{s+3}\max_{1\le l\le p_j-1}|k_l|^3\prod_{l=1}^{p_j-1}\LR{k_{l}}^3
}
for $1\le n\le 6$.
Thus, by (ii) of Lemma \ref{lem_ga4} with $p=p_j, m=1, d=s, c=s+3, b=3, a=3$, we obtain
\EQ{\label{E3.6}
&\Big|{\Ga}^{(p_j)}\big(M_n ;f,\ldots,f,g,f-g\big)-{\Ga}^{(p_j)}\big(M_n ;g,\ldots,g,g,f-g\big)\Big|\\
\lec &\|f-g\|_{H^{s}}\|g\|_{H^{s+3}} (\|f\|_{H^4}+\|g\|_{H^4})^{p_j-3}\\
\times&\{\|f-g\|_{H^7}(\|f\|_{H^4}+\|g\|_{H^4})+\|f-g\|_{H^4}(\|f\|_{H^7}+\|g\|_{H^7}) \}\\
\lec &\|f-g\|^2_{H^{s}}(\|f\|_{H^4}+\|g\|_{H^4})^{p_j-1}
+ \|f-g\|^2_{H^7}\|g\|^2_{H^{s+3}}(\|f\|_{H^7}+\|g\|_{H^7})^{p_j-3}
}
for $1\le n\le 6$.
By Lemmas \ref{lem_ga1}, we obtain
\EQ{\label{E3.7}
&\Big|{\Ga}^{(p_j)}\big(M_1 ;f,\ldots,f,f-g,f-g\big)-{\Ga}^{(p_j)}\big( M_{1,*} ;f,\ldots,f,f-g,f-g\big)\Big|\\
&\lec \|f\|_{H^4}^{p_j-2}\|f\|_{H^{8}}\|f-g\|^2_{H^s}.
}
By Lemmas \ref{lem_ga23}, we obtain
\EQ{\label{E3.8}
&\Big|{\Ga}^{(p_j)}\big(M_2 ;f,\ldots,f,f-g,f-g\big)-P_{N_j}(f)\|\p_x^{s+1}(f-g)\|^2_{L^2}\\
&-{\Ga}^{(p_j)}\big( M_{2,*} ;f,\ldots,f,f-g,f-g\big)\Big|\\
+&\Big|{\Ga}^{(p_j)}\big(M_4 ;f,\ldots,f,f-g,f-g\big)\\
&-{\Ga}^{(p_j)}\big( (s-1)M_{4,*} ;f,\ldots,f,f-g,f-g\big)\Big|\\
\lec & \|f\|_{H^4}^{p_j-2}\|f\|_{H^{8}}\|f-g\|^2_{H^s}.
}
For $n=3,5,6$, by Lemma \ref{byparts_lem}, 
\EQQ{
{\Ga}^{(p_j)}\big(M_n ;f,\ldots,f,f-g,f-g\big)={\Ga}^{(p_j)}\big(M_{n,*} ;f,\ldots,f,f-g,f-g\big)
}
where
\EQQ{
&M_{3,*}:=-\la_jc_j(ik_{p_j})^{s}(ik_{p_j+1})^{s}(ik_{(1,p_j-1)}) D_{a_j,b_j,c_j-1}/2,\\
&M_{5,*}:=-(s-1)\la_jb_j(ik_{(1,p_j-1)})^2(ik_{p_j})^{s}(ik_{p_j+1})^{s} D_{a_j,b_j-1,c_j}/2,\\
&M_{6,*}:=-(s-1)(s-2)\la_ja_j(ik_{(1,p_j-1)})^3(ik_{p_j})^{s}(ik_{p_j+1})^{s} D_{a_j-1,b_j,c_j}/4.
}
Since
\EQQ{
|M_{n,*}|\lec \LR{k_{p_j}}^s\LR{k_{p_j+1}}^s\max_{1\le l\le p_j-1}|k_l|^3\prod_{l=1}^{p_j-1}\LR{k_l}^3
}
for $n=3,5,6$, by the Sobolev inequality, we obtain
\EQ{\label{E3.9}
\Big|{\Ga}^{(p_j)}\big(M_n ;f,\ldots,f,f-g,f-g\big)\Big|
  \lec & \|f-g\|^2_{H^s}\|f\|_{H^7}\|f\|_{H^4}^{p_j-2}
}
for $n=3,5,6$.
Note that $(ik_{p_j})^{s+1}(ik_{p_j+1})^{s+1}M_{NR,j}=M_{1,*}+M_{2,*}+(s-1)M_{4,*}$. Therefore, collecting \eqref{E3.3}--\eqref{E3.9}, we conclude the desired result.
\end{proof}
\begin{prop}\label{multi_est2}
Let $s\in \N, s \ge 7$, $\e_1, \e_2\in [0,1]$ and $u\in L^\infty([0,T];H^s(\T)\cap H^8(\T))$ (resp. $v\in L^\infty ([0,T]; H^{s+4}(\T))$) be a solution to \eqref{es1} with $\e=\e_1$ (resp. $\e=\e_2$) on $[0,T]$.
Then, we have
\EQ{\label{est2}
& \Big|\frac{d}{dt} {\Ga}^{(p_j)}\Big(\frac{(ik_{p_j})^{s+1}(ik_{p_j+1})^{s+1}M_{NR,j}}{\Phi^{(p_j)}};u,\ldots,u,u-v,u-v \Big)\\
& \ \ \ \ \ \ -{\Ga}^{(p_j)}\big((ik_{p_j})^{s+1}(ik_{p_j+1})^{s+1}M_{NR,j};u,\ldots,u,u-v,u-v \big)\Big|\\
\le & \frac{\e_1}{4j_0}\|\p_x^{s+2} (u-v)\|_{L^2}^2+C|\e_1-\e_2|^2\|v\|^2_{H^{s+2}}\\
&+C\|u-v\|^2_{H^s}(1+\|u\|_{H^8}+\|v\|_{H^8})^{2(p_{\max}-1)}\\
&+C\|u-v\|^2_{H^8}(\|u\|_{H^s}+\|v\|_{H^s})^2(1+\|u\|_{H^8}+\|v\|_{H^8})^{2(p_{\max}-2)}\\
&+C\|u-v\|^2_{H^4}\|v\|^2_{H^{s+4}}(1+\|u\|_{H^8}+\|v\|_{H^8})^{2(p_{\max}-2)}.
}
\end{prop}
\begin{proof}
Let $\ti{u}(t):=U(-t)u(t), \ti{v}(t):=U(-t)v(t)$ where $U(t):=\F^{-1} e^{-t\phi(k)}\F_x$.
By the Leibniz rule, we have
\EQQ{
&\frac{d}{dt}{\Ga}^{(p_j)}\big((ik_{p_j})^{s+1}(ik_{p_j+1})^{s+1}M_{NR,j}/\Phi^{(p_j)};u,\ldots,u,u-v,u-v\big)\\
=&\frac{d}{dt}{\Ga}^{(p_j)}\big(e^{t\Phi^{(p_j)}}(ik_{p_j})^{s+1}(ik_{p_j+1})^{s+1}M_{NR,j}/\Phi^{(p_j)}; \ti{u},\ti{u},\ldots,\ti{u},\ti{u}-\ti{v},\ti{u}-\ti{v}\big)\\
=&I_{0}+I_{1}+\cdots+I_{p_j-1}+J_1+J_2,
}
where
\EQQ{
I_0:=&{\Ga}^{(p_j)}\big(\big(\p_t e^{t\Phi^{(p_j)}}\big)(ik_{p_j})^{s+1}(ik_{p_j+1})^{s+1}M_{NR,j}/\Phi^{(p_j)}; \ti{u},\ti{u},\ldots,\ti{u},\ti{u}-\ti{v},\ti{u}-\ti{v}\big)\\
I_1:=&{\Ga}^{(p_j)}\big(e^{t\Phi^{(p_j)}}(ik_{p_j})^{s+1}(ik_{p_j+1})^{s+1}M_{NR,j}/\Phi^{(p_j)}; \p_t\ti{u},\ti{u},\ldots,\ti{u},\ti{u}-\ti{v},\ti{u}-\ti{v}\big)\\
I_2:=&{\Ga}^{(p_j)}\big(e^{t\Phi^{(p_j)}}(ik_{p_j})^{s+1}(ik_{p_j+1})^{s+1}M_{NR,j}/\Phi^{(p_j)}; \ti{u},\p_t\ti{u},\ldots,\ti{u},\ti{u}-\ti{v},\ti{u}-\ti{v}\big)\\
&\vdots\\
I_{p_{j-1}}:=&{\Ga}^{(p_j)}\big(e^{t\Phi^{(p_j)}}(ik_{p_j})^{s+1}(ik_{p_j+1})^{s+1}M_{NR,j}/\Phi^{(p_j)}; \ti{u},\ti{u},\ldots,\p_t\ti{u},\ti{u}-\ti{v},\ti{u}-\ti{v}\big)\\
J_1:=&{\Ga}^{(p_j)}\big(e^{t\Phi^{(p_j)}}(ik_{p_j})^{s+1}(ik_{p_j+1})^{s+1}M_{NR,j}/\Phi^{(p_j)}; \ti{u},\ti{u},\ldots,\ti{u},\p_t(\ti{u}-\ti{v}),\ti{u}-\ti{v}\big)\\
J_2:=&{\Ga}^{(p_j)}\big(e^{t\Phi^{(p_j)}}(ik_{p_j})^{s+1}(ik_{p_j+1})^{s+1}M_{NR,j}/\Phi^{(p_j)}; \ti{u},\ti{u},\ldots,\ti{u},\ti{u}-\ti{v},\p_t(\ti{u}-\ti{v})\big).
}
Since
\EQQ{
I_0={\Ga}^{(p_j)}\big((ik_{p_j})^{s+1}(ik_{p_j+1})^{s+1}M_{NR,j};u,\ldots,u,u-v,u-v \big),
}
the left-hand side of \eqref{est2} is equal to $|I_1+\cdots +I_{p_j-1}+J_1+J_2|$.
We estimate only $|I_1|$ and $|J_1|$ because we can easily estimate $|I_2|,\ldots, |I_{p_j-1}|$ and $|J_2|$ in the same manner. 
Note that $\ti{u}, \ti{v}$ satisfy
\EQ{\label{E3.12}
\p_t \ti{u}=U(-t)(-\e_1\p_x^4 u+N(u)), \ \p_t \ti{v}=U(-t)(-\e_2\p_x^4 v+N(v)).
}
We substitute \eqref{E3.12} for $\p_t \ti{u}$ in $I_1$.
Then, we have
\EQQ{
I_1=&-\e_1 {\Ga}^{(p_j)}\Big(\frac{(ik_{p_j})^{s+1}(ik_{p_j+1})^{s+1}M_{NR,j}}{\Phi^{(p_j)}};\p_x^4 u,u,\ldots,u,u-v,u-v\Big)\\
&+ {\Ga}^{(p_j)}\Big(\frac{(ik_{p_j})^{s+1}(ik_{p_j+1})^{s+1}M_{NR,j}}{\Phi^{(p_j)}}; N(u),u,\ldots,u,u-v,u-v\Big).
}
By Lemma \ref{oscillation_est} and \eqref{eq_mnr}, it follows that
\EQQ{
\Big|\frac{(ik_{p_j})^{s+1}(ik_{p_j+1})^{s+1}M_{NR,j}}{\Phi^{(p_j)}}\Big| \lec |k_{p_j}|^s|k_{p_j+1}|^s\LR{k_1}^{1/2}\prod_{l=2}^{p_j-1}\LR{k_l}^3.
}
Therefore, by the Sobolev inequality,
\EQQ{
|I_1|&\lec \|u-v\|^2_{H^{s}}(\e_1\|\p_x^4 u\|_{H^2}\|u\|_{H^4}^{p_j-2}+\sum_{i=1}^{j_0}\|N_i(u)\|_{H^2}\|u\|_{H^4}^{p_j-2})\\
&\lec \|u-v\|_{H^s}^2(1+\|u\|_{H^6})^{2(p_{\max}-1)}.
}
We substitute \eqref{E3.12} for $\p_t (\ti{u}- \ti{v})$ in $J_1$ to have
\EQQ{
J_1=&{\Ga}^{(p_j)}\Big(\frac{(ik_{p_j})^{s+1}(ik_{p_j+1})^{s+1}M_{NR,j}}{\Phi^{(p_j)}};u,\ldots,u,-\e_1\p_x^4u+\e_2\p_x^4v+N(u)-N(v),u-v\Big)\\
=&-\e_1 {\Ga}^{(p_j)}\Big(\frac{(ik_{p_j})^{s+5}(ik_{p_j+1})^{s+1}M_{NR,j}}{\Phi^{(p_j)}};u,\ldots,u,u-v,u-v\Big)\\
&-(\e_1-\e_2) {\Ga}^{(p_j)}\Big(\frac{(ik_{p_j})^{s+5}(ik_{p_j+1})^{s+1}M_{NR,j}}{\Phi^{(p_j)}};u,\ldots,u,v,u-v\Big)\\
&+ \sum_{i=1}^{j_0}{\Ga}^{(p_j)}\Big(\frac{(ik_{p_j+1})^{s+1}M_{NR,j}}{\Phi^{(p_j)}}; u,\ldots,u,\p_x^{s+1}\big(N_i(u)-N_i(v)\big),u-v\Big)\\
=:&J_{1,1}+J_{1,2}+\sum_{i=1}^{j_0}K_i
}
By Lemma \ref{oscillation_est} and \eqref{eq_mnr},
\EQQ{
&\Big|\frac{(ik_{p_j})^{s+5}(ik_{p_j+1})^{s+1}M_{NR,j}}{\Phi^{(p_j)}}\Big|\\
\lec &|k_{p_j}|^{s+1}|k_{p_{j+1}}|^{s+1}\prod_{l=1}^{p_j-1} \LR{k_l}^3 \sim |k_{p_j}|^{s+2}|k_{p_{j+1}}|^{s}\prod_{l=1}^{p_j-1} \LR{k_l}^3.
}
Therefore, by the Sobolev inequality and Lemma \ref{lemGN}, we have
\EQQ{
|J_{1,1}|\le& C \e_1 \|u\|^{p_j-1}_{H^4}\|\p_x^{s+1}(u-v)\|_{L^2}^2\\
\le& C\e_1 \|u\|^{p_j-1}_{H^4}\|\p_x^s(u-v)\|_{L^2}\|\p_x^{s+2}(u-v)\|_{L^2}\\
\le& \frac{\e_1}{8j_0}\|\p_x^{s+2}(u-v)\|_{L^2}^2+C\e_1\|u\|^{2(p_j-1)}_{H^4}\|\p_x^s(u-v)\|^2_{L^2},
}
and
\EQQ{
|J_{1,2}|\lec  |\e_1-\e_2|\|u\|_{H^4}^{p_j-1}\|v\|_{H^{s+2}}\|u-v\|_{H^s}\lec  |\e_1-\e_2|^2\|v\|^2_{H^{s+2}} +\|u-v\|^2_{H^s}\|u\|_{H^4}^{2(p_j-1)}.
}
By Lemma \ref{Ki}, $|K_i|$ is bounded by the right-hand side of \eqref{est2}.
\end{proof}
\begin{prop}\label{multi_est3}
Let $q\in \N$, $\e_1,\e_2\in [0,1]$ and $u\in L^\infty([0,T];H^7(\T))$ (resp. $v\in L^\infty([0,T];H^4(\T))$) be a solution to \eqref{es1} with $\e=\e_1$ (resp. $\e=\e_2$) on $[0,T]$.
Then, we have
\EQ{\label{e3.71}
&\frac{d}{dt} \|u-v\|_{L^2}^2\\
\lec& |\e_1-\e_2|\|v\|_{L^2}\|u-v\|_{H^4}
+\|u-v\|^2_{H^4}\sum_{j=1}^{j_0}(\|u\|_{H^4}+\|v\|_{H^4})^{p_j-1},
}
and
\EQ{\label{e3.72}
\frac{d}{dt} \|u\|_{H^4}^{q} \lec \|u\|_{H^7}\sum_{j=1}^{j_0}\|u\|_{H^4}^{p_j+q-2}
}
on $[0,T]$.
\end{prop}
\begin{proof}
\EQQ{
&\big(\p_t (u-v), u-v\big)_{L^2}\\
=&\big(-\e_1\p_x^4u+\e_2\p_x^4v,u-v\big)_{L^2}+\big(N(u)-N(v),u-v\big)_{L^2}\\
=&-\e_1\big(\p_x^4(u-v),u-v\big)_{L^2}-(\e_1-\e_2)\big(\p_x^4v,u-v\big)_{L^2}+\sum_{j=1}^{j_0}\big(N_j(u)-N_j(v),u-v\big)_{L^2}\\
}
Thus,
\EQQ{
&\frac{1}{2}\frac{d}{dt}\|u-v\|_{L^2}^2\\
\le &|\e_1-\e_2|\|v\|_{L^2}\|\p_x^4(u-v)\|_{L^2}+\|u-v\|_{L^2}\sum_{j=1}^{j_0}\|N_j(u)-N_j(v)\|_{L^2}.
}
Therefore, by the Sobolev inequality, we obtain \eqref{e3.71}. Put $v\equiv 0$.
Then, by the same manner, we have
\EQ{\label{el2}
\frac{d}{dt}\|u\|_{L^2}^2 \lec \|u\|_{L^2}\sum_{j=1}^{j_0}\|u\|_{H^4}^{p_j}.
}
Since
\EQQ{
\big(\p_t \p_x^4 u, \p_x^4u\big)_{L^2}=-\e_1\big(\p_x^{8}u,\p_x^4u \big)_{L^2}+\sum_{j=1}^{j_0}\big(\p_x^4N_j(u), \p_x^4u\big)_{L^2},
}
by Lemma \ref{lemGN}, we have
\EQ{\label{eh4}
\frac{d}{dt}\|\p_x^4 u\|_{L^2}^2 \lec  \sum_{j=1}^{j_0}\|\p_x^4 N_j(u)\|_{L^2}\|\p_x^4 u\|_{L^2} \lec \|u\|_{H^7}\sum_{j=1}^{j_0}\|u\|_{H^4}^{p_j}.
}
From \eqref{el2} and \eqref{eh4}, we obtain
\EQQ{
2\|u\|_{H^4}\frac{d}{dt}\|u\|_{H^4}=\frac{d}{dt}\|u\|_{H^4}^2\lec
\|u\|_{H^7}\sum_{j=1}^{j_0}\|u\|_{H^4}^{p_j},
}
which imply \eqref{e3.72} with $q=1$ and $q=2$. We easily obtain \eqref{e3.72} with $q \ge 3$ from
\EQQ{
\frac{d}{dt}\|u\|_{H^4}^q =q \|u\|_{H^4}^{q-1}\frac{d}{dt}\|u\|_{H^4}.
}
\end{proof}
Now, we prove Proposition \ref{thm_energy_est}.
\begin{proof}[Proof of Proposition \ref{thm_energy_est}]
Put
\EQQ{
I_j:=&(-1)^s\big({N}_j(u)-N_j(v) ,\p_x^{2s} (u-v)\big)_{L^2}
+{P}_{N_j}(u)\|\p_x^{s+1} (u-v)\|_{L^2}^2\\
&+\frac{d}{dt}{\Ga}^{(p_j)}\Big( \frac{(ik_{p_j})^{s+1}(ik_{p_j+1})^{s+1}M_{NR,j}}{\Phi^{(p_j)}}; u,\ldots,u,u-v,u-v \Big).
}
Then, by Propositions \ref{multi_est}, \ref{multi_est2} and Lemma \ref{sobolev_energy}, we have
\EQ{\label{eq3.1.1}
\Big|\sum_{j=1}^{j_0} I_j \Big|\le &\frac{\e_1}{4}\|\p_x^{s+2} (u-v)\|_{L^2}^2+C|\e_1-\e_2|^2E_{s+2}(v)\\
+&CF_s(u,v)(1+E_8(u)+E_8(v))^{p_{\max}-1}\\
+&CF_8(u,v)(E_s(u)+E_s(v))(1+E_8(u)+E_8(v))^{p_{\max}-2}\\
+&CF_7(u,v)E_{s+4}(v)(1+E_8(u)+E_8(v))^{p_{\max}-2}.
}
By Proposition \ref{multi_est3} and Lemma \ref{sobolev_energy}, we have
\EQ{\label{eq3.1.2}
&\frac{d}{dt}\Big( \|u-v\|_{L^2}^2+C_s\|u-v\|_{L^2}^2\sum_{j=1}^{j_0}\|u\|_{H^4}^{s(p_j-1)}\Big)\\
\lec &\big(|\e_1-\e_2|\|v\|_{L^2}\|u-v\|_{H^4}+\|u-v\|_{H^4}^2(1+\|u\|_{H^4}+\|v\|_{H^4})^{p_{\max}-1}\big)(1+\|u\|_{H^4})^{s(p_{\max}-1)}\\
&+\|u-v\|^2_{L^2}\|u\|_{H^7}(1+\|u\|_{H^4})^{p_{\max}+s(p_{\max}-1)-2}\\
\lec & |\e_1-\e_2|^2E_1(v)+F_4(u,v)(1+E_7(u)+E_4(v))^{s(p_{\max}-1)}.
}
By integration by parts,
\EQQ{
&\big(\p_t \p_x^s(u-v), \p_x^s(u-v)\big)_{L^2}\\
=&-\big(\e_1\p_x^{s+4}u-\e_2\p_x^{s+4}v,\p_x^{s}(u-v)\big)_{L^2}+\sum_{j=1}^{j_0}(-1)^s\big(N_j(u)-N_j(v),\p_x^{2s}(u-v)\big)_{L^2}.
}
Thus,
\EQ{\label{eq3.1.3}
&\frac{1}{2}\frac{d}{dt}\|\p_x^s(u-v)\|_{L^2}^2+\e_1\|\p_x^{s+2}(u-v)\|_{L^2}^2\\
&-\sum_{j=1}^{j_0}(-1)^s\big(N_j(u)-N_j(v),\p_x^{2s}(u-v)\big)_{L^2}\\
=&-(\e_1-\e_2) \big(\p_x^{s+4}v,\p_x^{s}(u-v)\big)_{L^2}\le F_s(u,v)+|\e_1-\e_2|^2E_{s+4}(v).
}
Collecting \eqref{eq3.1.1}--\eqref{eq3.1.3}, we obtain the desired result.
\end{proof}

\section{Proofs of the main theorems}
In this section, we prove Theorems \ref{thm_nonparabolic}, \ref{thm_parabolic1}, \ref{thm_parabolic2} and \ref{thm_higher_order}.
The existence, uniqueness and continuous dependence results on $[-T,0]$ with $P_N \le 0$ in Theorems \ref{thm_nonparabolic} and \ref{thm_parabolic1} follows from the results on $[0,T]$ with $P_N \ge 0$ and the transform $t \to -t$.
Thus, we show only the results on $[0,T]$ with $P_N \ge 0$.
By the same reason we prove Theorem \ref{thm_parabolic2} only on $[0,T]$ with $P_N(\vp)<0$.
First, we show the uniqueness results in Theorems \ref{thm_nonparabolic} and \ref{thm_parabolic1}.
\begin{proof}[Proof of the uniqueness results]
Since $u_1,u_2 \in L^\infty([0,T];H^{12}(\T))$, by Lemma \ref{sobolev_energy}, there exist $M>0$ such that $\sup_{t\in [0,T]}(E_{12}(u_1(t))+E_{12}(u_2(t)))\le M$.
By Proposition \ref{thm_energy_est} with $\e_1=\e_2=0$, $s=8$, we have
\EQQ{
\frac{d}{dt}F_{8}(u_1(t),u_2(t))\lec F_{8}(u_1(t),u_2(t))(1+M)^{r(8)}
}
on $[0,T]$. Thus, by Lemma \ref{sobolev_energy} and Gronwall's inequality, we have
\EQQ{
\|u_1(t)-u_2(t)\|_{H^8}^2 \lec F_{8}(u_1(t),u_2(t))\le F_{8}(\vp,\vp)e^{C(1+M)^{r(8)}t}=0,
}
which implies $u_1(t)=u_2(t)$ on $[0,T]$.
\end{proof}
Next, we show the existence of a solution $u$ to \eqref{e1}--\eqref{e2} as a limit of the solutions $\{u_\e\}$ to \eqref{es1}--\eqref{es2} which are obtained by Proposition \ref{para_exist}.
In this process, it is important to ensure that $T_\e$ does not go to $0$ when $\e\to 0$.
For that purpose, we prepare a priori estimate below.
By \eqref{EF} and substituting $v \equiv 0$, we obtain the following energy inequality as a corollary of Proposition \ref{thm_energy_est}.
\begin{cor}[energy inequality]\label{cor_energy_est}
Assume that $s\in\N$, $s \ge 8$ and $\e \in [0,1]$.
Let $u_\e \in L^\infty([0,T];H^s(\T))$ be a solution to \eqref{es1} on $[0,T]$.
Then, it follows that
\EQQ{
\frac{d}{dt} E_{s}(u_\e(t))
+P_N(u_\e(t))\|\p_x^{s+1} u_\e(t)\|_{L^2}^2
\lec  E_s(u_\e(t))(1+E_{8}(u_\e(t)))^{r(s)}
}
on $[0,T]$, where $r(s):=s(p_{\max}-1) $ and the implicit constant depends on $s$ and does not depend on $\e, u_\e$, and $T$.
\end{cor}
\begin{lem}\label{pre_apriori}
Assume the assumption in Corollary \ref{cor_energy_est} and $P_N(u_\e)\ge 0$ on $[0,T_0]$.
Let $T_1:=\min\{T,T_0,(2C_1)^{-1}(1+E_8(\vp))^{-r(8)}\}$. Then, it follows that
\EQQ{
E_s(u_\e(t))+\int_0^t P_N(u_\e(t'))\|\p_x^{s+1} u_\e(t')\|_{L^2}^2\, dt' \le C_2 E_s(\vp),
}
on $[0,T_1]$, where $C_1, C_2$ are sufficiently large constants, $C_1$ does not depend on $\e, u_\e, s, \vp$ and $C_2=C_2(s,E_8(\vp))$ does not depend on $\e$ and $u_\e$.
\end{lem}
\begin{proof}
By Corollary \ref{cor_energy_est},
\EQQ{
\frac{d}{dt} E_{s}(u_\e(t))\le C(s) E_s(u_\e(t))(1+E_{8}(u_\e(t)))^{r(s)},
}
on $[0,T_1]$.
Let $s=8$. Then, we obtain
\EQ{\label{eq32.1}
E_{8}(u_\e(t))\le \frac{1+E_{8}(\vp)}{\big(1-C_1t(1+E_{8}(\vp))^{r(8)}\big)^{1/r(8)}}\lec 1+E_{8}(\vp)
}
on $[0,T_1]$. Combining Corollary \ref{cor_energy_est} and \eqref{eq32.1}, we have
\EQQ{
\frac{d}{dt} E_s(u_\e(t))+P_N(u_\e(t))\|\p_x^{s+1} u_\e(t)\|^2_{L^2}\lec E_s(u_\e(t))(1+E_{8}(\vp))^{r(s)}
}
on $[0,T_1]$.
Therefore, we obtain
\EQQ{
&E_s(u_\e(t))+\int_0^t P_N(u_\e(t'))\|\p_x^{s+1} u_\e(t')\|^2_{L^2} \, dt'\\
\le & E_s(\vp)e^{C(s)(1+E_{8}(\vp))^{r(s)} t}\le C_2(s,E_8(\vp))E_s(\vp)
}
on $[0,T_1]$.
\end{proof}
\begin{lem}\label{pre_apriori2}
Let $\e\in [0,1]$ and $u_\e\in L^\infty([0,T];H^{8}(\T))$ be a solution to \eqref{es1}--\eqref{es2} on $[0,T]$.
Assume that $K>0$ satisfies $\sup_{t\in [0,T]} E_8(u_\e(t)) \le K$
and $P_N(\vp)>0$. Then, there exists $T_+=T_+(K, P_N(\vp))$
such that $0<T_+\le T$ and
\EQ{\label{Pest}
\inf_{t \in [0,  T_+]}P_N(u_\e(t))  \ge  P_N(\vp)/2.
} 
\end{lem}
\begin{proof}
By \eqref{es1}, the Sobolev inequality and Lemma \ref{sobolev_energy},
\EQQ{
\sup_{t \in [0,  T_+]}\|\p_t u_\e(t)\|_{H^3} \lec \sup_{t \in [0,  T_+]}(1+\|u_\e(t)\|_{H^{8}})^{p_{\max}} \lec C(K).
}
Thus, by the Sobolev inequality and Lemma \ref{sobolev_energy}, we have
\EQQ{
\sup_{t \in [0,  T_+]} \Big|\frac{d}{dt}P_N(u_\e(t))\Big| \lec \sup_{t \in [0,  T_+]} \|\p_t u_\e(t)\|_{H^3}(1+\|u_\e(t)\|_{H^4})^{p_{\max}-2} \le C(K).
}
Put $T_+:=\min\{T, P_N(\vp)/2C(K)\}$. Then. by the mean value theorem,
\EQQ{
\sup_{t \in [0,  T_+]}|P_N(u_\e(t))-P_N(\vp)|\le C(K)T_+  \le P_N(\vp)/2.
}
\end{proof}
\begin{prop}[a priori estimate]\label{prop_apriori}
Assume the assumption in Corollary \ref{cor_energy_est}.
If $P_N \equiv 0$ or $P_N(\vp) >0$, then there exists $T_*=T_*(E_8(\vp), P_N(\vp))\in (0,T]$ such that
\EQS{
&\sup_{t\in [0,T_*]}\Big\{E_s(u_\e(t))+P_N(\vp) \int_0^t \|\p_x^{s+1} u_\e(t')\|_{L^2}^2\, dt'\Big\} \lec E_s(\vp),\label{apriori1}\\
&\inf_{t \in [0,T_*]}P_N(u_\e(t))  \ge  P_N(\vp)/2,\label{apriori2}
}
where the implicit constant does not depend on $\e, u_\e$ and may depend on $s, E_8(\vp)$.
\end{prop}
\begin{proof}
The case $P_N \equiv 0$ immediately follows from Lemma \ref{pre_apriori}.
For the proof of the case $P_N(\vp)>0$, we use the continuity argument.
Obviously, it follows that
\EQ{\label{cont}
P_N(u(t)) \ge P_N(\vp)/2 >0
}
at $t=0$.
Let $T_*:=\min\{ T,(2C_1)^{-1}(1+E_8(\vp))^{-r(8)},T_+(C_2(E_8(\vp))E_8(\vp),P_N(\vp)) \}$ and $0<t_* \le T_*$.
We assume $P_N(u(t))\ge 0$ on $[0,t_*]$.
Then, by Lemma \ref{pre_apriori}, $E_8(u_\e(t))\le C_2(E_8(\vp))E_8(\vp)$ on $[0,t_*]$.
By Lemma \ref{pre_apriori2} with $K=C_2(E_8(\vp))E_8(\vp)$, we obtain \eqref{cont} on $[0,t_*]$.
Since $P_N(u(t))$ is continuous, we conclude that \eqref{cont} holds on $[0,T_*]$, that is \eqref{apriori2}.
By Lemma \ref{pre_apriori}, we obtain \eqref{apriori1} on $[0,T_*]$.
\end{proof}
When $P_N \equiv 0$ or $P_N(\vp) >0$ and $\vp\in H^{s}(\T)$ with $s \ge 12$,
we have a priori estimate of $E_{s}(u_\e)$ on $[0,T_*]$.
Therefore, by Proposition \ref{thm_energy_est} and Lemma \ref{sobolev_energy}, we can easily conclude that $\{u_\e\}$ is a Cauchy sequence in $C([0,T_*];H^{s-4}(\T))$ and the limit is a solution to \eqref{e1}--\eqref{e2}.
However, this argument is not enough for our purpose because the regularity of the solution obtained by this argument is weaker than that of initial data.
We introduce Bona-Smith's approximation argument (\cite{BonaS}) to obtain a  solution in $C([0,T_*]; H^s(\T))$.
\begin{defn}
For $\y \in (0,1], s\ge 0$, $f\in H^s(\T)$, we put
\EQQ{
\ha{J_{\y,s} f}(k):=\exp (-\y (1+|k|^2)^{s/2})  \ha{f}(k).
} 
\end{defn}
For the proof of the following lemma, see Lemma 6.14 in \cite{Iorio}.
\begin{lem}\label{lem_BS}
Let $0\le j\le s$, $0\le l$ and $f\in H^s(\T)$. Then, $J_{\y,s} f \in H^\infty(\T)$ satisfies
\EQQ{
&\|J_{\y,s} f -f\|_{H^s} \to 0 \ \ (\y\to 0),\\
&\|J_{\y,s} f -f \|_{H^{s-j}} \lec \y^{j/s} \|f\|_{H^s}, \ \ \|J_{\y,s}f\|_{H^{s-j}}\lec \|f\|_{H^{s-j}},\\
&\|J_{\y,s} f\|_{H^{s+l}} \lec \y^{-l/s}\|f\|_{H^s}.
}
\end{lem}
Next, we show the existence results in Theorems \ref{thm_nonparabolic}, \ref{thm_parabolic1}.
\begin{proof}[Proof of the existence results]
Fix $s \ge 13$. Let $\y=\e \in (0,1]$ and $\vp_\y:=J_{\y,s}\vp \in H^\infty(\T)$.
By Proposition \ref{para_exist}, we have the solution $u_{\e,\y}\in C([0,T_\e); H^\infty(\T))$ to \eqref{es1} with initial data $\vp_\y$.
Since $P_N \equiv 0$ or $P_N(\vp_\y) \to P_N(\vp)> 0$ as $\y\to 0$, by Proposition \ref{prop_apriori}, Lemma \ref{sobolev_energy} and Lemma \ref{lem_BS}, there exists $T_{*,\y}=T_{*,\y}(E_8(\vp_\y),P_N(\vp_\y))\in (0,T_\e]$ such that
\EQS{
\begin{split}
\sup_{t\in [0,T_{*,\y}]} E_{12}(u_{\e,\y}(t)) &\lec E_{12}(\vp_\y) \lec \| \vp_{\y}\|^2_{H^{12}}+(1+\|\vp_{\y}\|_{H^4})^{12p_{\max}}\\
& \lec \|\vp\|^2_{H^{12}}+(1+\|\vp\|_{H^4})^{12p_{\max}},\label{e4.21}
\end{split}\\
\begin{split}
\sup_{t\in [0,T_{*,\y}]} E_{s}(u_{\e,\y}(t)) &\lec E_{s}(\vp_\y) \lec \| \vp_{\y}\|^2_{H^{s}}+(1+\|\vp_{\y}\|_{H^4})^{sp_{\max}}\\
& \lec \|\vp\|^2_{H^{s}}+(1+\|\vp\|_{H^4})^{sp_{\max}},\label{e4.212}
\end{split}\\
\begin{split}
\sup_{t\in [0,T_{*,\y}]} E_{s+4}(u_{\e,\y}(t)) &\lec E_{s+4}(\vp_\y) \lec \| \vp_{\y}\|^2_{H^{s+4}}+(1+\|\vp_\y\|_{H^4})^{(s+4)p_{\max}}\\
& \lec \y^{-8/s}\|\vp\|^2_{H^s}+  (1+\|\vp\|_{H^4})^{(s+4)p_{\max}},\label{e4.22}
\end{split}\\
\inf_{t \in [0,T_{*,\y}]}P_N(u_{\e,\y}(t)) & \ge  P_N(\vp_\y)/2 \gec P_N(\vp),\label{e4.23}
}
for sufficiently small $\y>0$.
Since $P_N(\vp_\y) \gec P_N(\vp)$ and $\|\vp_\y\|_{H^8} \lec \|\vp\|_{H^8}$, there exists $T_*=T_*(\|\vp\|_{H^{8}}, P_N(\vp)) \in (0,T_{*,\y}]$ for sufficiently small $\y>0$.
Thus, by Proposition \ref{thm_energy_est} with $s=8$ and \eqref{monotone}, we have
\EQQ{
\frac{d}{dt}F_8(u_{\e_1,\y_1}(t),u_{\e_2,\y_2}(t))
\le  C(\|\vp\|_{H^{12}})(F_8(u_{\e_1,\y_1}(t),u_{\e_2,\y_2}(t))+|\e_1-\e_2|^2),
}
on $[0,T_*]$.
Here we assumed $0<\e_1=\y_1\le \e_2 =\y_2\le 1$.
By Gronwall's inequality, there exists $T=T(\|\vp\|_{H^{12}}, P_N(\vp))\in (0,T_*]$ such that
\EQ{\label{F8}
&\sup_{t\in [0,T]}F_8(u_{\e_1,\y_1}(t),u_{\e_2,\y_2}(t))\\
\lec& F_8(\vp_{\y_1},\vp_{\y_2})+|\e_1-\e_2|^2\\
\lec& \|\vp_{\y_1}-\vp_{\y_2}\|^2_{H^8}+\|\vp_{\y_1}-\vp_{\y_2}\|_{L^2}^2(1+\|\vp_{\y_1}\|_{H^4})^{8(p_{\max}-1)}+\y_2^2\\
\lec& \y_2^{2(s-8)/s}(1+\|\vp\|_{H^s})^{8p_{\max}}.
}
Here, we used Lemma \ref{sobolev_energy} and Lemma \ref{lem_BS}.
By Proposition \ref{thm_energy_est}, \eqref{monotone}, \eqref{e4.21}, \eqref{e4.212}, \eqref{e4.22} and \eqref{F8},
\EQQ{
&\frac{d}{dt} F_{s}(u_{\e_1,\y_1}(t),u_{\e_2,\y_2}(t))
+P_N(u_{\e_1,\y_1}(t))\|\p_x^{s+1} (u_{\e_1,\y_1}(t)-u_{\e_2,\y_2}(t))\|_{L^2}^2\\
\lec&  F_s(u_{\e_1,\y_1}(t),u_{\e_2,\y_2}(t))(1+\|\vp\|^2_{H^{12}}+(1+\|\vp\|_{H^4})^{12p_{\max}})^{r(s)}\\
+&\big\{(\y_2^{2(s-8)/s}(1+\|\vp\|_{H^s})^{8p_{\max}}(1+\|\vp\|^2_{H^{12}}+(1+\|\vp\|_{H^4})^{12p_{\max}})^{p_{\max}-2}+|\e_1-\e_2|^2\big\}\\
&\times (\y_2^{-8/s}\|\vp\|^2_{H^s}+\|\vp\|_{H^4}^{(s+4)p_{\max}})\\
\lec& C(s,\|\vp\|_{H^s})\big(F_s(u_{\e_1,\y_1}(t),u_{\e_2,\y_2}(t))+\y_2^{2(s-12)/s}\big).
}
By Gronwall's inequality, Lemma \ref{sobolev_energy} and \eqref{e4.23}, we conclude
\EQ{\label{FS}
&\sup_{t\in [0,T]}\|u_{\e_1,\y_1}(t)-u_{\e_2,\y_2}(t))\|^2_{H^s}
+P_N(\vp)\int_0^T\|\p_x^{s+1} (u_{\e_1,\y_1}(t)-u_{\e_2,\y_2}(t))\|_{L^2}^2\, dt\\
\lec &\sup_{t\in [0,T]}F_{s}(u_{\e_1,\y_1}(t),u_{\e_2,\y_2}(t))
+\int_0^T P_N(u_{\e_1,\y_1}(t))\|\p_x^{s+1} (u_{\e_1,\y_1}(t)-u_{\e_2,\y_2}(t))\|_{L^2}^2\, dt\\
\lec &F_{s}(\vp_{\y_1},\vp_{\y_2})+\y_2^{2(s-12)/s}\\
\lec &\|\vp_{\y_1}-\vp_{\y_2}\|_{H^s}^2+\|\vp_{\y_1}-\vp_{\y_2}\|_{L^2}^2(1+\|\vp_{\y_1}\|_{H^4})^{s(p_{\max}-1)}+\y_2^{2(s-12)/s}\to 0
}
as $0<\y_1\le \y_2 \to 0$.
Thus, $\{u_{\e,\y}\}_{\e=\y}$ is a Cauchy sequence in $C([0,T];H^s(\T))$.
By using \eqref{es1}, we also have that $\{\p_t u_{\e,\y}\}_{\e=\y}$ is a Cauchy sequence in $C([0,T]; H^{s-5}(\T))$.
Therefore, there exists the limit $u\in C([0,T];H^s(\T))\cap C^1([0,T];H^{s-5}(\T))$ as $\e=\y\to 0$, which satisfies \eqref{e1}--\eqref{e2}.
Finally, we assume $P_N(\vp)>0$ and prove $u\in C^\infty((0,T]\times \T)$. 
We have $u\in L^2([0,T];H^{s+1}(\T))$ by the second term of the left-hand side of \eqref{FS}.
Let $0<\de<T$. Since $u(t) \in H^{s+1}(\T)\, a.e. \, t\in [0,T]$, we can choose $t_0$ such that $0<t_0<\de/2$ and $u(t_0)\in H^{s+1}(\T)$.
Applying the same argument above with initial data $\vp:=u(t_0)\in H^{s+1}(\T)$,
we conclude $u\in C([t_0,T];H^{s+1}(\T)) \cap  L^2([t_0,T];H^{s+2}(\T))$.
We can choose $t_1$ such that $\de/2<t_1<\de(1/2+1/4)$ and $u(t_1) \in H^{s+2}(\T)$.
Applying the same argument above with initial data $\vp:=u(t_1)\in H^{s+2}(\T)$,
we conclude $u\in C([t_1,T];H^{s+2}(\T)) \cap  L^2([t_1,T];H^{s+3}(\T))$.
By repeating this process, we conclude $u \in C([\de,T];H^\infty(\T))$. By using \eqref{e1} and the Sobolev inequality, we also have $u \in C^\infty([\de,T]\times \T)$.
Because we can take arbitrarily small $\de>0$, we conclude $u \in C^\infty((0,T]\times \T)$.
\end{proof}
Next, we show the continuous dependence results in Theorems \ref{thm_nonparabolic}, \ref{thm_parabolic1}.
\begin{proof}[Proof of the continuous dependence results]
Taking $0< \e_1= \y_1 \to 0$ in \eqref{FS}, we have
\EQ{\label{FSlimit}
\sup_{t\in [0,T]}\|u(t)-u_{\e,\y}(t))\|^2_{H^s}
\lec C(s,\|\vp\|_{H^s}) (\|\vp-\vp_{\y}\|_{H^s}^2+\y^{2(s-12)/s})
}
for any sufficiently small $\e = \y>0$.
We show that for any $\de_0 >0$, there exists $N_0\in \N$ such that
if $j\ge N_0$, then $\sup_{t\in [0,T]}\|u(t)-u^j(t)\|_{H^s} \lec \de_0$.
There exists $N_1\in\N$ such that for any $j\ge N_1$,
$\|J_{\y,s}(\vp^j-\vp)\|_{H^s}\le \|\vp^j-\vp \|_{H^s}\le \de_0$.
Thus, by Lemma \ref{lem_BS},
\EQQ{
\|\vp^j-\vp^j_{\y}\|_{H^s}\le \|\vp^j-\vp \|_{H^s}+\|\vp -\vp_{\y}\|_{H^s}+\|\vp_{\y}-\vp^j_{\y}\|_{H^s}
\lec \de_0
}
for any $j\ge N_1$ if we take $\y>0$ sufficiently small.
Thus, by \eqref{FSlimit}
\EQ{\label{e4.1}
\sup_{t\in [0,T]}\|u(t)-u_{\e,\y}(t))\|_{H^s}+\sup_{t\in [0,T]}\|u^j(t)-u^j_{\e,\y}(t))\|_{H^s}
\lec \de_0
}
for any $j\ge N_1$ and sufficiently small $\y$.
Here we fix sufficiently small $\y>0$ such that \eqref{e4.1} holds. By Proposition \ref{para_exist}, there exists $N_2\in\N$ such that for any $j\ge N_2$, 
$\|u_{\e,\y}(t)-u^j_{\e,\y}(t)\|_{H^s}\le \de_0$.
Therefore, for any $j\ge N_0:=\max\{N_1,N_2\}$, we obtain
\EQQ{
&\sup_{t\in [0,T]} \|u(t)-u^j(t)\|_{H^s}\\
\le& \sup_{t\in [0,T]}\big(\|u(t)-u_{\e,\y}(t)\|_{H^s} +\|u_{\e,\y}(t)-u^j_{\e,\y}(t)\|_{H^s}+\|u^j_{\e,\y}(t)-u^j(t)\|_{H^s}\big)\lec \de_0.
}
\end{proof}
Next, we show Theorem \ref{thm_parabolic2}.
\begin{proof}[Proof of Theorem \ref{thm_parabolic2}]
We prove it by contradiction. Assume that there exists a solution $u\in C([0,T];H^{13}(\T))$ to \eqref{e1}--\eqref{e2}.
We take $\de>0$ sufficiently small. Then,
we have $P_N(u(t)) \le P_N(\vp)/2<0$ on $[0,\de]$ and the lifetime $T=T(P_N(u(\de)),\|u(\de)\|_{H^{12}})$ in Theorem \ref{thm_parabolic1} satisfies $T \ge \de$.
Applying Theorem \ref{thm_parabolic1} with initial data $u(\de)$, we have $u \in C^\infty ( [0,\de)\times \T )$ since $P_N(u(\de))<0$.
This contradicts the assumption $\vp \not\in C^\infty(\T)$.
\end{proof} 
Finally we give the outline of the proof of Corollary \ref{thm_higher_order}.
\begin{proof}[Outline of the proof of Corollary \ref{thm_higher_order}]
We consider the following parabolic regularized equation:
\EQS{
\begin{split}
&(\p_t +\e\p_x^4+\ga_0\p_x^{2j+1}+\ga_1 \p_x^{2j-1}+\cdots+ \ga_j \p_x) u_\e(t,x)\\
&\hspace*{9em}= N(\p_x^3 u_\e, \p_x^2 u_\e, \p_xu_\e, u_\e), \quad (t,x)\in [0,T_\e) \times \T,
\end{split}\label{es_high}
}
instead of \eqref{es1}. We can prove Proposition \ref{para_exist} replaced \eqref{es1} with \eqref{es_high} in the same manner. 
We replace the definition of $\phi$ with $\phi(k):=i((-1)^j \ga_0 k^{2j+1}+(-1)^{j-1}\ga_1 k^{2j-1}+\cdots+\ga_j k)$.
Then, Lemma \ref{oscillation_est} holds for this new definition of $\phi$.
Precisely, in the same manner, we have much better estimate 
\EQ{\label{ineq_oscillation_higher}
|\Phi^{(p)}(\vec{k}^{(p)})| \gec |k_p|^{2j}|k_{(1,p-1)}|\sim|k_{p+1}|^{2j}|k_{(1,p-1)}|.
}
However, \eqref{ineq_oscillation} is enough for our proof.
Though $\Phi^{(p)}, E_s, F_s$ depend on the definition of $\phi$, 
all propositions and lemmas in Section 2 and Section 3 with this new definition of $\phi$ hold if we replace \eqref{es1} with \eqref{es_high}.
Therefore, The proofs in Section 4 are valid if we replace \eqref{e1} and \eqref{es1} with \eqref{he1} and \eqref{es_high}. Thus, we obtain Corollary \ref{thm_higher_order}.
\end{proof}

\end{document}